\documentclass[reqno]{amsart} 
\usepackage{enumerate,amsmath,amssymb} 
\usepackage{pstricks,pst-node,pst-coil,pst-plot} 
\usepackage{hyperref}

\newtheorem{theorem}{Theorem}[section]
\newtheorem{proposition}[theorem]{Proposition}

\newtheorem*{theorem*}{Claim}
\newtheorem{lemma}[theorem]{Lemma}

\theoremstyle{definition}
\newtheorem{definition}   [theorem]  {Definition}

\theoremstyle{definition}
\newtheorem{remark}   [theorem]  {Remark}

\theoremstyle{definition}
\newtheorem{example}   [theorem]  {Example}

\newcommand{\definedas}{\mathrel{\raise.095ex\hbox{\rm :}\mkern-5.2mu=}}

\let\oldmarginpar\marginpar
\renewcommand\marginpar[1]{\-\oldmarginpar[\raggedleft\tiny #1]%
{\raggedright\tiny #1}}

\newcounter{mnotecount}[section]

\DeclareMathOperator{\tr}{tr}

\DeclareMathOperator{\arcsinh}{arcsinh}

\newcommand{\bR}{\mathbb{R}}

\newcommand{\ftil}{\widetilde{f}}

\newcommand{\bU}{\mathbb{U}}

\newcommand{\bQ}{\mathbb{Q}}

\newcommand{\bH}{\mathbb{H}}

\newcommand{\cL}{\mathcal{L}}

\newcommand{\cQ}{\mathcal{Q}}

\newcommand{\cR}{\mathcal{R}}

\newcommand{\gbar}{\bar{g}}

\renewcommand{\hbar}{\overline{h}}

\newcommand{\gtil}{\tilde{g}}
\newcommand{\Jtil}{\widetilde{J}}
\newcommand{\Jbar}{\bar{J}}

\newcommand{\pitil}{\widetilde{\pi}}
\newcommand{\pibar}{\bar{\pi}}
\newcommand{\mutil}{\widetilde{\mu}}
\newcommand{\mubar}{\bar{\mu}}

\newcommand{\Ebrev}{\breve{E}}
\newcommand{\Pbrev}{\breve{P}}

\DeclareMathOperator{\divg}{div}

\newcommand{\rlie}{\mathring{\mathcal{L}}}

\newcommand{\ric}{\mathrm{Ric}}

\newcommand{\scal}{\mathrm{Scal}}

\newcommand{\hess}{\mathrm{Hess}}

\begin{document} 

\author{Mattias Dahl}
\address{Institutionen f\"or Matematik \\
Kungliga Tekniska H\"ogskolan \\
100 44 Stockholm \\
Sweden} 
\email{dahl@kth.se}

\author{Anna Sakovich}
\address{Max-Planck-Institut f\"ur Gravitationsphysik 
(Albert-Einstein-Institut) \\
Am M\"uhlenberg 1, 14476 Potsdam, Germany}
\address{Matematiska Institutionen, Uppsala University, Box 480, 751 06 Uppsala, Sweden}
\email{anna.sakovich@math.uu.se}

\title[A density theorem for asymptotically hyperbolic initial data]
{A density theorem for\\ 
asymptotically hyperbolic initial data\\ 
satisfying the dominant energy condition}

\begin{abstract}
When working with asymptotically hyperbolic initial data sets for general relativity it is convenient to assume certain simplifying properties. We prove that the subset of initial data sets with such properties is dense in the set of physically reasonable asymptotically hyperbolic initial data sets. More specifically, we show that an asymptotically hyperbolic initial data set with non-negative local energy density can be approximated by an initial data set with strictly positive local energy density and a simple structure at infinity, while changing the mass arbitrarily little. This is achieved by suitably modifying the argument used by Eichmair, Huang, Lee and Schoen in the asymptotically Euclidean case.
\end{abstract}

\date{\today}

\maketitle

\section{Introduction}

In general relativity Einstein's equations read 
\begin{equation} \label{eqEinstein}
\ric^{\gamma} - \tfrac{1}{2}\scal^{\gamma} \gamma = \mathrm{T}.
\end{equation}
Here $\ric^{\gamma}$ and $\scal^{\gamma}$ denote respectively the Ricci tensor 
and the scalar curvature of a spacetime $(\mathcal{M},\gamma)$, and the 
symmetric divergence-free 2-tensor $\mathrm{T}$ is the stress-energy 
tensor of the spacetime. A spacetime $(\mathcal{M},\gamma)$ satisfying 
\eqref{eqEinstein} is said to obey the dominant energy condition if for 
any future directed timelike vector $\nu$ the vector 
$-\mathrm{T}(\nu,\cdot)^\sharp$ is either future directed timelike or null. 
This condition means that the energy density of $(\mathcal{M},\gamma)$ is 
non-negative and that the energy cannot travel faster than the speed of 
light. 

Let $(M,g)$ be a Riemannian submanifold of the spacetime 
$(\mathcal{M},\gamma)$ satisfying the Einstein equations with unit normal 
denoted by $\eta$ and second fundamental form denoted by $K$. In this case 
$(M,g)$ can be viewed as a ``constant time slice'' of $(\mathcal{M},\gamma)$.
The dominant energy condition for $(\mathcal{M},\gamma)$ at points of $M$
is equivalent to the inequality $\mu \geq |J|_g$ everywhere on $M$.
Here the energy density $\mu \definedas \mathrm{T}(\eta, \eta)$ and the 
momentum density $J \definedas \mathrm{T}(\eta, \cdot)$ can be computed 
from $g$ and $K$ using the {\em constraint equations} 
\begin{align}
-2 \mu &= -\scal^g-(\tr^g K)^2+|K|_g^2, \label{eqConstraints1} \\
J &= \divg^g K-d (\tr^g K). \label{eqConstraints2}
\end{align}

By an {\em initial data set} for the Einstein equations we will mean a 
triple $(M,g,K)$ consisting of a Riemannian $n$-manifold $(M,g)$ and a 
symmetric 2-tensor field $K$ defined on $M$. We will say that $(M,g,K)$ satisfies 
the {\em dominant energy condition} if $\mu\geq |J|_g$ holds everywhere on 
$M$, where $\mu$ and $J$ are defined through 
\eqref{eqConstraints1}--\eqref{eqConstraints2}. By the above discussion, 
this means that $(M,g,K)$ arises as a constant time slice of a spacetime 
satisfying the dominant energy condition.

An initial data set $(M,g,K)$ is said to be {\em asymptotically Euclidean} 
if outside some compact set $M$ is diffeomorphic to the complement of a 
ball in Euclidean space $\bR^n$, and if under this diffeomorphism $g$ approaches 
the Euclidean metric $\delta$ and $K$ approaches zero sufficiently fast at infinity. 
For asymptotically Euclidean initial data sets {\em asymptotic charge integrals} at
infinity can be defined. They are integrals which arise as boundary terms
when integrating the constraint operator
\[
\Phi: (g,K) \mapsto 
\left( -\scal^g - (\tr^g K)^2 + |K|_g^2, \; \divg^g K-d (\tr^g K) \right)
\]
against elements in the kernel of $D\Phi^*_{(\delta,0)}$, which correspond 
to Killing vectors of the Minkowski spacetime. In particular, this gives 
rise to the Arnowitt-Deser-Misner energy $E$ and linear momentum $P$. The
{\em positive mass conjecture} asserts that $E \geq |P|_g$ provided that
the dominant energy condition holds. In particular, $E\geq 0$ is expected
to hold under the same assumption, which is the statement of the
{\em positive energy conjecture}. An excellent overview of both conjectures can be found in the recent book by Lee \cite{LeeGR}.

For many applications in mathematical general relativity it is an advantage to work with initial data sets which have simple asymptotics at infinity. For example, an asymptotically Euclidean initial data set $(M,g,K)$ is said to have {\em harmonic asymptotics} if in asymptotically Euclidean coordinates at infinity we have
\[
g=u^{\tfrac{4}{n-2}} \delta, \qquad
\pi \definedas K-(\tr^g K) g 
= u^{\tfrac{2}{n-2}} (\cL_Y \delta - \divg^\delta Y),
\]
where $\cL$ denotes the Lie derivative, and the positive function $u$ and 
the vector field $Y$ are such that 
\[
u(x) = 1 + A |x|^{2-n} + O(|x|^{1-n}) , 
\qquad Y_j(x) = B_j |x|^{2-n} + O(|x|^{1-n}),
\]
for $A \in \bR$ and $B \in \bR^n$. In this case the Arnowitt-Deser-Misner 
energy and linear momentum can be easily read off from the asymptotic 
expansions of $u$ and $Y$ at infinity, namely,
\[
E= 2A, \qquad P_{j} = - \tfrac{n-2}{n-1} B_{j}.
\]

Further, many arguments can be simplified by working with initial data sets with strictly positive local energy density, that is such that the {\em strict dominant energy condition} $\mu > |J|_g$ is satisfied. This condition is preserved under ``small'' perturbations of the initial data set, whereas the standard dominant energy condition $\mu \geq |J|_g$ might get violated by a perturbation. In \cite[Theorem~18]{SpacetimePMT}, Eichmair, Huang, Lee, and Schoen prove that an asymptotically Euclidean initial data set satisfying dominant energy condition can be slightly perturbed to an initial data set with harmonic asymptotics which obeys strict dominant energy condition while changing the energy $E$ and the linear momentum $P$ arbitrarily little. That is, the set of asymptotically Euclidean initial data sets with these preferred properties is dense in the set of asymptotically Euclidean initial data sets satisfying the dominant energy condition. This result is used in the proof of the positive mass theorem by the above authors \cite[Theorem~1]{SpacetimePMT}, and it is also required for the proof of the positive energy conjecture in dimension $n=3$ by Schoen and Yau \cite{PMT2}, and for its extension to dimensions $3\leq n \leq 7$ by Eichmair \cite{EichmairReduction}. Another application is the analysis of the geometry and topology of initial data sets with horizons, see \cite{TopCensorshipID}.

The goal of the current paper is to prove the analogue of this density 
result for asymptotically hyperbolic initial data sets. Roughly 
speaking, an initial data set $(M,g,K)$ is asymptotically hyperbolic if the Riemannian metric $g$ approaches the hyperbolic metric $b$ on hyperbolic space $\bH^n$ in a chart covering everything outside a compact set. 
For $K$, there are two natural choices: either $K \to 0$ at infinity 
(as for spacelike totally geodesic hypersurfaces in asymptotically 
anti-de~Sitter spacetimes) or $K \to g$ at infinity 
(as for ``hyperboloidal'' hypersurfaces in asymptotically Minkowski spacetimes)
\footnote{We note that asymptotically hyperbolic initial data sets can also be modeled on totally umbilic hyperbolic slices of de~Sitter spacetime, for example \cite{LiangZhang}.}.
In this paper we adopt the second approach and consider ``hyperboloidal'' 
initial data, see Definition~\ref{defAHdata}. Then similar considerations 
as in the asymptotically Euclidean case give rise to the notion of 
{\em mass} for asymptotically hyperbolic initial data, which is a linear 
functional on a certain finite dimensional vector space. 

The main result of this paper is that a given asymptotically hyperbolic initial data set satisfying the dominant energy condition can be approximated by an initial data set with {\em conformally hyperbolic asymptotics} in the sense of Definition~\ref{defConfHypID} which obeys the strict dominant energy condition, while changing the value of the mass functional by an arbitrarily small amount. In fact, assuming sufficient regularity, one can construct coordinates at infinity in which the approximating initial data set has {\em Wang's asymptotics} in the sense of Definition~\ref{defWangData}. In particular, we prove the following

\begin{theorem}
Let $(M,g,K)$ be an asymptotically hyperbolic initial data set of type 
$(k+1, \alpha, \tau, \tau_0)$ for $0<\alpha<1$, $\tfrac{n}{2} < \tau < n$ and
$\tau_0>0$ satisfying the dominant energy condition $\mu \geq |J|_g$. Then for every $\varepsilon>0$ there exists
an asymptotically hyperbolic initial data set $(M,\bar{g},\bar{K})$ of type
$(k, \alpha, n , \tau_0')$ for some $\tau'_0>0$ with Wang's asymptotics satisfying 
the strict dominant energy condition 
\[
\bar{\mu} > |\bar{J}|_{\bar{g}}
\]
and such that the mass functionals $\mathcal{M}$ of $(M,g,K)$ and  $\overline{\mathcal{M}}$ of $(M,\bar{g},\bar{K})$ satisfy
\[
|\mathcal{M} (V) - \overline{\mathcal{M}} (V)|
< \varepsilon
\]
for any $V \in \mathcal{N}$, where $\mathcal{N}$ is the linear space spanned by restrictions of coordinate functions of Minkowski spacetime to the upper unit hyperboloid.
\end{theorem}

The applications of our results are similar to those of \cite[Theorem~18]{SpacetimePMT}. In particular, the results are used in the proofs of the positive mass theorem for asymptotically hyperbolic manifolds by Chru\'sciel and Delay \cite{ChruscielDelayPMT} and for asymptotically hyperbolic initial data sets by the second author \cite{AHJang}. They can also prove useful for establishing  other geometric inequalities for asymptotically hyperbolic initial data, such as those discussed in \cite{ChaKhuriSakovich}.

The paper is organized as follows. The definition of mass and its continuity with respect to the initial data set is discussed in Section~\ref{secPreliminaries}. In Section~\ref{secStrictDEC} we show that a given asymptotically hyperbolic initial data set satisfying the dominant energy condition can be perturbed slightly to satisfy the strict dominant energy condition, while changing the mass arbitrarily little. Then in Section~\ref{secConfHyp} we make a further perturbation to conformally hyperbolic asymptotics, while preserving the strict dominant energy condition.
In Section~\ref{secWang} we prove a density result concerning asymptotically hyperbolic initial data sets that have Wang's asymptotics. We also discuss how one can switch to Wang's asymptotics given the approximating initial data set constructed in Section~\ref{secConfHyp}.
Finally, in Section~\ref{secRemarks} we give comments on possible extensions
of the results of this paper. Some supplementary results concerning
differential operators on asymptotically hyperbolic manifolds are
contained in Appendices \ref{appFredholm}, \ref{secAbove}, and \ref{secUC}. 

\subsection*{Acknowledgements}

We want to thank Romain Gicquaud and Lan-Hsuan Huang for helpful discussions and comments related to the work in this paper. Further, we want to thank two anonymous referees for their careful reading and many insightful remarks helping to improve the paper.

\section{Preliminaries}
\label{secPreliminaries}

\subsection{Asymptotically hyperbolic initial data}
\label{secDefAH}

We denote the hyperbolic space of dimension $n$ by $\bH^n$ and the hyperbolic 
metric by $b$. We choose a point in $\bH^n$ as the origin. In polar 
coordinates around this point we have $b = dr^2 + \sinh^2r \, \sigma$ on
$(0,\infty) \times S^{n-1}$, where $\sigma$ is the round metric on the unit
sphere $S^{n-1}$ and $r$ is the distance to the origin. The open ball of
radius $R$ centered at the origin is denoted by $B_R$, and its closure is
denoted by $\overline{B}_R$.

We first give the definition of an asymptotically hyperbolic Riemannian
manifold.

\begin{definition} \label{defAHmanifold}
We say that $(M, g)$ is a {\em $C^{l, \beta}_\tau$-asymptotically hyperbolic 
manifold} for a non-negative integer $l$, $0 \leq \beta <1$ and $\tau> 0$, 
if there exists a compact set $K_0 \subset M$, $R_0>0$ and a diffeomorphism
\[
\Psi: M \setminus K_0 \to \bH^n \setminus \overline{B}_{R_0}, 
\]
such that $\Psi_* g - b \in C^{l, \beta}_{loc}(\bH^n; S^2\bH^n)$ and
\[
\left\|\Psi_* g - b \right\|_{C^{l, \beta}_\tau(\bH^n\setminus B_{R_0}; S^2 \bH^n)} 
\definedas 
\sup_{x \in \bH^n\setminus B_{R_0+1}} e^{\tau r(x)} 
\left\|\Psi_* g - b\right\|_{C^{l, \beta}(B_1(x); S^2\bH^n)}
< \infty.
\]
\end{definition}

The diffeomorphism $\Psi$ introduced in this definition is called a 
{\em chart at infinity} for the asymptotically hyperbolic manifold. 

Let $(M,g)$ be a $C^{l, \beta}_\tau$-asymptotically hyperbolic manifold for 
$l$, $\beta$, $\tau$ as in Definition~\ref{defAHmanifold}. Suppose that
$u$ is a locally integrable section of a geometric tensor bundle $E$
(see \cite[Chapter~3]{LeeFredholm} for the definition of geometric tensor
bundles) over $M\setminus K_0$. In this case we say that
$u \in W_\delta^{k,p}(M \setminus K_0)$ for $0\leq k \leq l$, and
$1<p<\infty$ if 
\[
\left\|u \right\|_{W^{k, p}_\delta(M \setminus K_0)} \definedas 
\left\| e^{\delta r} \Psi_* u \right\|_{W^{k, p}(\bH^n \setminus B_{R_0})} < \infty.
\]
Note also the following equivalent definition of the 
$W^{k, p}_\delta(M \setminus K_0)$ norm,
\[
\left\|u \right\|_{W^{k, p}_\delta(M \setminus K_0)} 
=\sum_{0\leq j\leq k} \|e^{\delta r} \nabla^j (\Psi_* u)\|_{L^p(\bH^n \setminus B_{R_0})}.
\]
Building on these definitions, it is straightforward to define
{\em weighted Sobolev spaces} $W^{k,p}_\delta (M)$ for any $\delta$, 
$0\leq k \leq l$, and $1<p<\infty$. 
The {\em weighted H\"older spaces} $C^{k, \alpha}_\delta(M)$ are 
defined in a similar fashion. We say that 
$u \in C_\delta^{k,\alpha}(M \setminus K_0)$ for 
$0\leq k+\alpha \leq l + \beta$, and $0\leq \alpha <1$ if
\[
\left\| u \right\|_{C^{k, \alpha}_\delta(M \setminus K_0)} 
\definedas \sup_{x \in \bH^n\setminus B_{R_0+1}} e^{\tau r(x)} 
\left\| \Psi_* u \right\|_{C^{k, \alpha}(B_1(x))} < \infty.
\]
The following equivalent definition of the
$C^{k, \alpha}_\delta(M \setminus K_0)$ norm is often useful,
\[
\left\|u \right\|_{C^{k, \alpha}_\delta(M \setminus K_0)} 
= \sum_{0\leq j \leq k} \sup_{\bH^n \setminus B_{R_0}} |e^{\delta r} \nabla^j (\Psi_* u)| 
+ \|e^{\delta r} \nabla^k(\Psi_* u)\|_{C^{0,\alpha}(\bH^n \setminus B_{R_0})}.
\]
Again, these definitions can be extended to define weighted H\"older 
spaces $C^{k,\alpha}_\delta (M)$.

The weighted Sobolev and H\"older spaces that we have just defined are 
analogues of the respective spaces defined by J. Lee in \cite{LeeFredholm} 
on conformally compact manifolds. It is easy to check that standard facts 
such as embedding theorems, the Rellich lemma, and density theorems
hold for these spaces and that the statements of these results repeat 
verbatim the respective statements in \cite{LeeFredholm}. In particular 
Lemma~3.6 and Lemma~3.9 of \cite{LeeFredholm} hold for $W^{k,p}_\delta (M)$ 
and $C^{k,\alpha}_\delta (M)$ as defined above and we will refer to 
\cite{LeeFredholm} for these results throughout the text.

It is straightforward to check that classical interior elliptic regularity 
as formulated in \cite[Lemma~4.8]{LeeFredholm} holds for asymptotically
hyperbolic manifolds and weighted function spaces as defined in Section
\ref{secDefAH}. In Appendix \ref{appFredholm} we show that improved elliptic
regularity \cite[Proposition~6.5]{LeeFredholm} holds in the current setting. 
As a consequence, Fredholm theory for geometric elliptic operators on
asymptotically hyperbolic manifolds in the sense of 
Definition~\ref{defAHmanifold} can be established, since the proof of 
\cite[Theorem~C]{LeeFredholm} can be adapted. The reader is referred to 
Appendix \ref{appFredholm} for details.

We can now give the definition of an asymptotically hyperbolic initial
data set. Recall that the energy density $\mu$ and the momentum density
$J$ are defined via the constraint equations 
\eqref{eqConstraints1}--\eqref{eqConstraints2}. 

\begin{definition} \label{defAHdata}
A triple $(M, g, K)$ is an {\em asymptotically hyperbolic initial data set
of class $(k, \alpha, \tau)$} for $k\geq 2$, $0 \leq \alpha <1$ and $\tau > 0$ if 
\begin{itemize}
\item
$(M,g)$ is a $C^{k,\alpha}_\tau$-asymptotically hyperbolic manifold in the
sense of Definition~\ref{defAHmanifold},
\item
a symmetric 2-tensor $K$ is such that $K-g \in C^{k-1,\alpha}_{\tau}(M;S^2 M)$.
\end{itemize}
If, in addition, $(\mu,J)\in C^{k-2,\alpha}_{n+\tau_0}$ for some $\tau_0>0$
then $(M, g, K)$ is an {\em asymptotically hyperbolic initial data set
of class $(k,\alpha,\tau,\tau_0)$}. 
\end{definition}

Abusing notation slightly we may summarize the content of this definition
as follows: $(M,g,K)$ is an asymptotically hyperbolic initial data set of
class $(k,\alpha, \tau)$ for $0 \leq \alpha <1$ and $\tau> 0$ if 
$(g-b,K-g) \in C^{k,\alpha}_\tau \times C^{k-1,\alpha}_\tau$. In this case
$(\mu,J)\in C^{k-2,\alpha}_\tau$. The necessity for the faster decay
$(\mu,J)\in C^{k-2,\alpha}_{n+\tau_0}$ will become clear in Section
\ref{secCharges}.

Given an asymptotically hyperbolic initial data set $(M,g,K)$ of class
$(\alpha,\tau)$ it is convenient to set
\[
\pi \definedas (K - g) - \tr^g (K - g) g.
\]
Note that $\pi \in C^{k-1,\alpha}_\tau$ and that $\pi$ contains the same
information as $K$, since $K = \pi + g - \tfrac{1}{n-1}(\tr ^g \pi)g$.
It is therefore equivalent to work with $(M,g,\pi)$ as an initial data
set, and in this paper we will only work with initial data sets given in this
form.

In terms of $(g,\pi)$ the constraint equations
\eqref{eqConstraints1}--\eqref{eqConstraints2} are written as
\begin{align*}
- 2 \mu &=
- (\scal^g + n(n-1)) + 2 \tr^g \pi - \frac{(\tr^g \pi)^2}{n-1} + |\pi|^2_g, \\
J &= \divg^g \pi. 
\end{align*}
By the {\em constraint map} we mean the map 
\begin{equation} \label{eqConstraintMap2}
\Phi:(g,\pi) \mapsto
\left( - (\scal^g + n(n-1)) + 2 \tr^g \pi - \frac{(\tr^g \pi)^2}{n-1}
+ |\pi|^2_g, \divg^g \pi\right). 
\end{equation}

Finally, we define initial data sets with conformally hyperbolic
asymptotics. Recall that the {\em conformal Killing operator} $\rlie$
is defined by
\[
(\rlie_Y g)_{ij}
= \nabla_i Y_j + \nabla_j Y_i - \tfrac{2}{n}(\divg^g Y)g_{ij},
\]
that is, $(\rlie_Y g)_{ij}$ is the trace-free part of the Lie derivative
$(\cL_Y g)_{ij} = \nabla_i Y_j + \nabla_j Y_i$.

\begin{definition} \label{defConfHypID}
We say that an initial data set $(M, g, \pi)$ has {\em conformally
hyperbolic asymptotics} if there exists a compact set $K_0$, a radius
$R_0>0$, and a diffeomorphism
\[
\Psi: M \setminus K_0 \to \bH^n \setminus \overline{B}_{R_0}, 
\]
such that
\[
\Psi_* g
= (1+v)^{\tfrac{4}{n-2}} b, \quad \Psi_* \pi
= (1+v)^{\tfrac{2}{n-2}} \rlie_Y b, 
\]
where the function $v$ and the components of the 1-form $Y$ can be written in the form 
\begin{equation} \label{eqConfHyp}
\begin{split}
v &= v_0 e^{-nr} +v_1, \\
Y_r &= (Y_0)_r e^{-n r} + (Y_1)_r,  \\
Y_\varphi &= (Y_0)_\varphi e^{- (n-1) r} + (Y_1)_\varphi,
\end{split}
\end{equation}
where $\varphi$ refers to a coordinate on $S^{n-1}$, $(v_0,Y_0) \in C^{k,\alpha}_{loc}$ is independent of $r$ and 
$(v_1,Y_1) \in C^{k,\alpha}_{n+1}$ for $k\geq 2$ and $0 \leq \alpha <1$.
\end{definition}

\subsection{The mass functional for asymptotically hyperbolic initial data}
\label{secCharges}

In this section we review the concept of mass in the asymptotically 
hyperbolic setting and discuss the continuity of mass with respect to the 
initial data. We first recall how the asymptotic charge integrals are 
defined, following Michel \cite{MichelMass}.

Let $(M,g,\pi)$ be an asymptotically hyperbolic initial data set of type 
$(k,\alpha, \tau)$ for $k \geq 2$, $0\leq \alpha<1$ and $\tau>0$, and let $\Psi$ be the 
chart at infinity as in Definition~\ref{defAHmanifold}. Clearly, in this 
case we have $e \definedas \Psi_* g-b \to 0$ and 
$\eta \definedas \Psi_*\pi \to 0$ at infinity. Let the constraint map 
$\Phi$ be defined by \eqref{eqConstraintMap2}. Since $\Phi(b,0)=0$, 
linearization gives us 
\begin{equation} \label{eqLinear0}
\Phi(\Psi_*(g,\pi)) = D\Phi|_{(b,0)} (e,\eta) + \mathcal{Q}(e,\eta),
\end{equation}
where $\mathcal{Q}(e,\eta)$ is a remainder term of second order. 
For any function $V$ and 1-form $\varpi$ there is a 1-form 
$\mathbb{U}_{(V,\varpi)} (e,\eta)$ such that
\[
\langle D\Phi|_{(b,0)} (e,\eta), (V,\varpi) \rangle 
= \divg^b \mathbb{U}_{(V,\varpi)} (e,\eta) 
+ \langle (e,\eta), D\Phi^*_{(b,0)}(V,\varpi) \rangle,
\]
where $D\Phi^*_{(b,0)}$ is the formal adjoint of $D\Phi|_{(b,0)}$. Here 
$\langle \cdot,\cdot \rangle$ denotes the inner product induced by $b$ on 
geometric tensor bundles over $\bH^n$. Contracting \eqref{eqLinear0} with 
$(V,\varpi) \in \ker D\Phi^*_{(b,0)}$ we obtain 
\begin{equation} \label{eqLinearContracted}
\langle \Phi(\Psi_*(g,\pi)), (V,\varpi) \rangle 
= \divg^b \mathbb{U}_{(V,\varpi)} (e,\eta) 
+ \langle \mathcal{Q}(e,\eta), (V,\varpi) \rangle.
\end{equation}
In this way we assign to every $(V,\varpi) \in \ker D\Phi^*_{(b,0)}$ 
the {\em charge integral}
\[ \begin{split}
\mathbb{Q}_{(V,\varpi)} (g,\pi) 
&\definedas 
\lim_{R\to \infty} \int_{S_R} \mathbb{U}_{(V,\varpi)} (e,\eta) (\nu) \, d\mu^b,
\end{split} \]
where $\nu$ is the outer unit normal of the $(n-1)$-dimensional sphere
$S_R$ in $\bH^n$. 
 
The structure of the kernel of $D\Phi^*_{(b,0)}$ is well understood, see
Moncrief \cite{Moncrief}. Namely, $(V,\varpi^\sharp)$ corresponds to the
 normal-tangential (or lapse-shift) decomposition of the restriction along 
the unit hyperboloid of a Killing vector field of Minkowski spacetime. 
In other words, $(V,\varpi^\sharp)$ is a {\em Killing initial data} 
(or {\em KID}) for Minkowski spacetime given on the unit hyperboloid. 

In particular, we have $(V,-dV) \in \ker D\Phi^*_{(b,0)}$ for
$V \in \mathcal{N}$, where the vector space $\mathcal{N}$ is spanned by 
the functions
\[
V_{(0)} = \cosh r, \quad
V_{(1)} = x^1 \sinh r, \quad \dots, \quad
V_{(n)} = x^n \sinh r
\]
expressed in polar coordinates on $\bH^n = (0,\infty) \times S^{n-1}$.
Here $x^1,\dots, x^n$ are the coordinate functions on $\bR^n$ restricted
to $S^{n-1}$.

For these KIDs we have the following result.

\begin{proposition} \label{propWellDefMass}
Let $(M,g,\pi)$ be an asymptotically hyperbolic initial data set of type
$(k,\alpha, \tau,\tau_0)$ for $k\geq 2$, $0\leq \alpha <1$, $\tau > \tfrac{n}{2}$, and
$\tau_0>0$. Then for every $V \in \mathcal{N}$ the charge integral
$\bQ_{(V,-dV)}(g,\pi)$ is well-defined and can be computed by the formula
\begin{equation} \label{eqMassQ}
\begin{split}
&\mathbb{Q}_{(V,-dV)}(g,\pi)\\ 
&\quad=
\lim_{R \to \infty} \int_{S_R}
\left(V (\operatorname{div}^b e- d \tr^b e)
+ (\tr^b e) dV - (e+2\eta) (\nabla^b V, \cdot)
\right)(\nu) \, d\mu^b.
\end{split}
\end{equation}
\end{proposition}
\begin{proof}
Integrating \eqref{eqLinearContracted} over $\bH^n \setminus B_{R_0}$ and
using the divergence theorem we obtain
\[ \begin{split}
&\bQ_{(V,-dV)}(g,\pi) \\
&\quad = 
\int_{\bH^n \setminus B_{R_0}} 
\langle \Phi(\Psi_*(g,\pi)) - \mathcal{Q}(e,\eta), (V,-dV) \rangle \, d\mu^b 
+ \int_{S_{R_0}} \bU_{(V,-dV)}(e,\eta) (\nu)\, d\mu^b.
\end{split} \]
Estimating $\mathcal{Q}(e,\eta)$ as in \cite[Equation (12)]{MichelMass} 
and using our assumptions on the decay of the initial data, we see that 
the integral over $\bH^n \setminus B_{R_0}$ converges, 
hence $\bQ_{(V,-dV)}(g,\pi)$ is well-defined. 

We refer to \cite[Section~IV.2.B]{MichelMass} and references therein for the
derivation of the formula \eqref{eqMassQ}.
\end{proof}

\begin{definition}
Let $(M,g,\pi)$ be an asymptotically hyperbolic initial data set. Then
the {\em mass} of $(M,g,\pi)$ is the linear functional
$\mathcal{M}_{(g,\pi)}: \mathcal{N} \to \bR$ given by
\[
\mathcal{M}_{(g,\pi)}(V)
= \tfrac{1}{2(n-1)\omega_{n-1}} \mathbb{Q}_{(V,-dV)}(g,\pi),
\]
where $\omega_{n-1}$ denotes the volume of the unit sphere $(S^{n-1},\sigma)$.
\end{definition}

This is the same as the expression for the Bondi mass obtained by 
Chru\'sciel, Jesierski, and {\L}{\c{e}}ski in
\cite{ChruscielJezierskiLeski}, under asymptotic decay conditions that
however do not allow for gravitational radiation. See \cite{ChruscielNagy}, \cite{CH}, 
and \cite{MichelMass} for discussions on coordinate covariance.

As Proposition~\ref{propWellDefMass} shows, the mass functional is well
defined for asymptotically hyperbolic initial data sets of type
$(k,\alpha,\tau,\tau_0)$ for $k\geq 2$, $0\leq \alpha < 1$, $\tau > \tfrac{n}{2}$,
and $\tau_0>0$. It is also straightforward to check that the mass 
functional is trivial for asymptotically hyperbolic initial data sets 
of type $(\alpha,\tau)$ with $\tau > n$. The following are two examples 
of the ``critical'' case $\tau=n$. 

\begin{example}
The Anti-de~Sitter Schwarzschild Riemannian metric is given by
\[
g_{\rm AdSS} =
\frac{d\rho^2}{1 + \rho^2 - \frac{2m}{\rho^{n-2}}} + \rho^2 \sigma
\]
on $[a, \infty) \times S^{n-1}$, where the inner radius $a$ depends on $m$,
see for example \cite[Appendix A]{DGSsmallmass}. It can be realized as
an \emph{umbilic} (that is, $g=K$) asymptotically hyperbolic initial data 
set for Schwarzschild spacetime, see Brendle and Wang \cite{BrendleWang}.
In this case 
\[
\mathcal{M}(V_{(0)}) = m, \qquad \text{and} \qquad
\mathcal{M}(V_{(i)}) = 0, 
\]
for $i=1, \dots, n$, where $m$ coincides with the mass parameter of the 
Schwarzschild metric.
\end{example}

\begin{example}
For initial data sets with conformally hyperbolic asymptotics as in
Definition~\ref{defConfHypID} it is not complicated to compute that 
\[
\mathcal{M} (V_{(0)})
=
\tfrac{2(n+1)}{(n-2)\omega_{n-1}} \int_{S^{n-1}} v_0 \,d\mu^\sigma
+ \tfrac{2(n+1)}{n \omega_{n-1}} \int_{S^{n-1}} (Y_0)_r \,d\mu^\sigma,
\]
and
\[
\mathcal{M} (V_{(i)})
=
\tfrac{2(n+1)}{(n-2)\omega_{n-1}} \int_{S^{n-1}} x^i v_0 \, d\mu^\sigma
+ \tfrac{2(n+1)}{n\omega_{n-1}} \int_{S^{n-1}} x^i (Y_0)_r \,d\mu^\sigma
\]
for $i=1, \dots, n$.
\end{example}

Concluding this section, we confirm that the mass is continuous as a
function of asymptotically hyperbolic initial data sets of type
$(k,\alpha, \tau, \tau_0)$, where $k\geq 2$, $0\leq \alpha< 1$, $\tau > \tfrac{n}{2}$,
and $\tau_0>0$. For simplicity, the charts at infinity are suppressed
in the statement of the result and in the proof.

\begin{proposition} \label{propContinuity}
Let $(g,\pi)$ and $(\bar{g}, \bar{\pi})$ be asymptotically hyperbolic 
initial data sets of type $(k,\alpha, \tau, \tau_0)$ for $k\geq 2$, $0\leq \alpha< 1$,
$\tau > \tfrac{n}{2}$, and $\tau_0>0$. Let $(\mu,J)$ and $(\bar{\mu},\bar{J})$ denote the respective energy and momentum densities
defined via the constraint equations \eqref{eqConstraints1}--\eqref{eqConstraints2}. 
 Given $\varepsilon>0$ there exists
$\delta>0$ depending only on $(g,\pi)$ and $\varepsilon$, such that if 
\begin{equation} \label{decayforcont1}
\| g-\bar{g} \|_ {C^{2}_\tau} \leq \delta , \quad
\|\pi-\bar{\pi} \|_ {C^{1}_\tau} \leq \delta ,
\end{equation}
and 
\begin{equation} \label{decayforcont3}
\|(\mu,J)-(\bar{\mu},\bar{J}) \|_ {C^{0}_{n+\tau_0}} \leq \delta, 
\end{equation}
then for any $V \in \{V_{(0)}, V_{(1)}, \ldots, V_{(n)}\}$ we have
\[
\left|
\mathcal{M}_{(g,\pi)}(V) - \mathcal{M}_{(\bar{g},\bar{\pi})}(V)
\right|
\leq \varepsilon.
\]
\end{proposition}

\begin{proof}
Fix $R \geq R_0$. Arguing as in the proof of Proposition~\ref{propWellDefMass}
we find that
\[ \begin{split}
& 2(n-1)\omega_{n-1} 
\left( \mathcal{M}_{(g,\pi)}(V) - \mathcal{M}_{(\bar{g},\bar{\pi})}(V) \right)\\
&\qquad=
\int_{\bH^n \setminus B_{R}} 
\langle \Phi(g,\pi) -\Phi(\bar{g}, \bar{\pi}), (V,-dV) \rangle \, d\mu^b \\
&\qquad \qquad
-\int_{\bH^n \setminus B_{R}}
\langle \mathcal{Q}(e,\eta) -\mathcal{Q}(\bar{e},\bar{\eta}), (V,-dV) \rangle
\, d\mu^b \\
&\qquad\qquad
+\int_{S_R} \left( 
\bU_{(V,-dV)}(e,\eta) - \bU_{(V,-dV)}( \bar{e},\bar{\eta})
\right) (\nu)\, d\mu^b.
\end{split} \]
Now suppose that $(g,\pi)$ is fixed and that $\delta$ and 
$(\bar{g},\bar{\pi})$ are such that \eqref{decayforcont1} and 
\eqref{decayforcont3} hold. Then by assumption \eqref{decayforcont3} the 
absolute value of the first integral over $\bH^n \setminus B_{R}$ is 
bounded by $C \delta$ for some $C>0$ depending only on $(g,\pi)$. The same 
is true for the second integral over $\bH^n \setminus B_{R}$ by assumption 
\eqref{decayforcont1} combined with the fact that the remainder term 
$\mathcal{Q}(e,\eta)$ in \eqref{eqLinear0} is at least quadratic in 
$e$ and $\eta$ and their derivatives of order up to 2 and 1 respectively. 
As for the inner boundary integral, we see that its absolute value is 
bounded by $C\delta e^{(n-\tau) R}$ for $C>0$ depending only on $(g,\pi)$. 
From this it is clear that $\delta$ can be chosen so that the sum of the 
absolute values of these three integrals is less than $\varepsilon$.
\end{proof}

\section{Perturbation to strict inequality in the dominant energy condition}
\label{secStrictDEC}

This section is devoted to the following result. 

\begin{theorem} \label{thStrictDEC}
Let $(M,g,\pi)$ be an asymptotically hyperbolic initial data set of type
$(k,\alpha, \tau, \tau_0)$ for $k\geq 2$, $0<\alpha<1$, $\tfrac{n}{2} < \tau < n$, 
and $\tau_0 > 0$. Assume that $(M,g,\pi)$ satisfies the dominant energy 
condition, $\mu \geq |J|_{g}$. Then for every $\varepsilon>0$ there exists an asymptotically hyperbolic initial data set $(\bar{g}, \bar{\pi})$, with the energy and momentum density denoted by $(\mubar,\Jbar)$, of type $(k, \alpha,\tau,\tau_0')$ for some $\tau_0'>0$ such that
\[
\|g-\bar{g}\|_{C^{k,\alpha}_{\tau}} < \varepsilon,
\qquad
\|\pi - \bar{\pi}\|_{C^{k-1,\alpha}_\tau} < \varepsilon,
\]
and the strict dominant energy condition 
\[
\bar{\mu} > (1+\gamma) |\bar{J}|_{\bar{g}}
\]
holds for a constant $\gamma>0$, and
\[
\left|
\mathcal{M}_{(g,\pi)}(V) - \mathcal{M}_{(\bar{g},\bar{\pi})}(V)
\right|< \varepsilon
\]
for $V\in \{V_{(0)},V_{(1)},\ldots,V_{(n)}\}$. 
\end{theorem}

The argument follows \cite[Proof of Theorem~22]{SpacetimePMT}. In simple 
terms it can be described as follows. We would like to choose
symmetric 2-tensors $h$ and $w$ so that the perturbed initial data
$\bar{g} = g + t h$ and $\bar{\pi} = \pi + t w$ satisfies
$\bar{\mu} > |\bar{J}|_{\bar{g}}$ for sufficiently small $t>0$. From the 
Taylor expansion
$\Phi(\bar{g},\bar{\pi}) = \Phi(g,\pi) + t D\Phi|_{(g,\pi)}(h,w) + O(t^2)$,
we see that $\bar{\mu} = \mu + \frac{t}{2} f + O(t^2)$ and
$\bar{J} = J + t X + O(t^2)$, where $(-f,X) = D\Phi|_{(g,\pi)} (h,w)$.
Further,
\begin{equation} \label{eqSloppy}
\begin{split}
|\bar{J}|_{\bar{g}}^2
&=
\bar{g}^{ij} \bar{J}_i \bar{J}_j \\
&=
(g^{ij} - t h^{ij} + O(t^2)) (J_i + t X_i + O(t^2)) (J_j + t X_j + O(t^2)) \\
&=
|J|_g^2 + t(2 X^j - h^{ij} J_i)J_j + O(t^2),
\end{split}
\end{equation}
where indices are raised using the metric $g$. Hence if we set
$X^j = \tfrac{1}{2} h^{ij} J_i$ then $|\bar{J}|_{\bar{g}} = |J|_g + O(t^2)$, 
as long as the decay of $|J|_g^2$ at infinity is not faster than that of 
the $O(t^2)$ term in the last line of \eqref{eqSloppy}. This leads to the 
expectation that $\bar{\mu} > |\bar{J}|_{\bar{g}}$ will be achieved if we 
can find a pair $(h,w)$ such that $D \Phi|_{(g,\pi)} (h,w) = (-f,X)$, where 
$X^j = \tfrac{1}{2} h^{ij} J_i$ and $f>0$. Indeed, in this case we have
\[
\bar{\mu} - |\bar{J}|_{\bar{g}} 
= \mu - |J|_g + t f + O (t^2) \geq t f + O(t^2) > 0 
\]
provided that the $O(t^2)$ term above decays at least as fast as $f$ at 
infinity. 

However, $D \Phi|_{(g,\pi)}$ is not a determined elliptic operator 
(see for example Delay \cite{Delay}), and this makes it difficult to ensure that the 
solutions of the equation $D \Phi|_{(g,\pi)} (h,w) = (-f,X)$ will have 
good asymptotic behaviour. This problem can be overcome by combining the 
above considerations with a certain construction introduced by Corvino 
and Schoen in their proof of the density result in
\cite[Theorem~1]{CorvinoSchoen}. The idea is similar in spirit to the 
conformal method of solving the constraint equations (see for example 
\cite[Section~4.1]{BartnikIsenberg}) and is based on the observation that 
by suitably choosing a first order differential operator $\mathcal{D}$ 
 one can ensure that the linearization at $(1,0)$ of the operator 
\begin{equation} \label{eqElliptic}
(u, Y) \mapsto 
\Phi \left(u^{\tfrac{4}{n-2}} g, u^{\tfrac{2}{n-2}}(\pi + \mathcal{D}Y)\right),
\end{equation}
is a second order elliptic operator with nice properties.

We begin the proof of Theorem~\ref{thStrictDEC} with some preliminaries. 
Set $\kappa \definedas \tfrac{4}{n-2}$.
For $(u-1, Y) \in C^{2,\alpha}_{\tau}$ we let
\begin{equation} \label{eqDeformedID}
\gtil = u^\kappa g,
\qquad \text{and} \qquad
\pitil = u^{\kappa/2} (\pi + \rlie_Y g),
\end{equation}
where $\rlie$ is the conformal Killing operator described in
Section~\ref{secDefAH}. Our choice of the
operator $\mathcal{D} = \rlie$ in \eqref{eqElliptic} is motivated by 
the fact that the \emph{vector Laplacian} $\Delta_L = \divg \rlie$ is 
a well-known elliptic operator on asymptotically hyperbolic manifolds 
whose Fredholm properties (see Appendix \ref{appFredholm}) fit nicely 
into the context of the current argument. Let $\mutil$ and $\Jtil$ be
the energy and momentum densities of $(\gtil, \pitil)$ computed via the
constraint equations \eqref{eqConstraints1}--\eqref{eqConstraints2} and
consider the operator
\[
T(u,Y) = (- 2 u^\kappa \mutil, u^{\kappa/2} \Jtil ).
\]
This conformal rescaling of the constraint equations is needed to ensure 
that the dominant energy condition scales correctly when we pass to the 
deformed initial data set \eqref{eqDeformedID}, see \eqref{eqBoundMu} and 
\eqref{eqBoundJ} below.

It is straightforward to check that
\begin{equation} \label{eqRescConstraints}
\begin{split}
-2 u^\kappa \mutil
&=
\tfrac {4(n-1)}{n-2} u^{-1} \Delta^g u - \scal^g - n(n-1) u^{\kappa}
+ 2 u^{\kappa/2} \tr^g \pi - \tfrac{1}{n-1} (\tr^g \pi)^2 \\
&\qquad
+ \left( |\pi|_g^2 + 2 \langle \pi, \rlie_Y g \rangle
+ |\rlie_Y g|_g^2 \right), \\
u^{\kappa/2} \Jtil_j
&=
(\Delta_L Y + \divg ^g \pi)_j
+ \tfrac{2(n-1)}{n-2} u^{-1} (\pi + \rlie_Y g )^k_j \nabla_k u
- \tfrac{2}{n-2} u^{-1} \nabla_j u \tr^g \pi,
\end{split}
\end{equation}
for $j=1,2, \dots, n$. Consequently, the linearization
of $T$ at $(1,0)$ is
\begin{equation} \label{eqLinearT} 
\begin{split}
DT|_{(1,0)} (v,Z)
&=
\Big( \tfrac{4(n-1)}{n-2} (\Delta^g v - n v) + \tfrac{4}{n-2}(\tr^g \pi) v
+ 2 \langle \pi, \rlie_Z g \rangle, \\
&\qquad
(\Delta_L Z)_j + \tfrac{2(n-1)}{n-2}\pi^k_j \nabla_k v
- \tfrac{2}{n-2}(\tr^g \pi) \nabla_j v \Big),
\end{split} 
\end{equation}
for $j=1,2, \dots, n$. The following lemma concerns Fredholm properties 
of the operator $DT|_{(1,0)}$.

\begin{lemma} \label{lemDT}
If $(M,g,\pi)$ is an asymptotically hyperbolic initial 
data set of type $(k,\alpha,\tau)$ for $k\geq 2$, $0<\alpha<1$ and $\tau>0$
then $DT|_{(1,0)}$ is a Fredholm operator with index zero in the following
cases:
\begin{itemize}
\item
as a map $C^{l,\beta}_\delta \to C^{l-2,\beta}_\delta$ for $2\leq l \leq k$, $0 < \beta \leq \alpha$, 
$-1<\delta<n$, 
\item
as a map $W^{l,p}_\delta \to W^{l-2,p}_\delta$ for $2\leq l \leq k$,  $1 < p< \infty$, 
$-1 < \delta + \tfrac{n-1}{p} < n $. 
\end{itemize}
\end{lemma}
\begin{proof}
We give the proof in the case of weighted H\"older spaces, the case
of weighted Sobolev spaces is treated similarly. Write
$DT|_{(1,0)} = P_0 + P_1$, where
\[
P_0: (v,Z) \mapsto 
\left(\tfrac{4(n-1)}{n-2} (\Delta^g v - n v), \Delta_L Z\right),
\]
and 
\[
P_1: (v,Z) \mapsto
\left(\tfrac{4}{n-2}(\tr^g \pi) v + 2\langle \pi, \rlie_Z g \rangle,
\tfrac{2(n-1)}{n-2}\pi^k_j \nabla_k v
- \tfrac{2}{n-2}(\tr^g \pi) \nabla_j v \right).
\]
Here $P_0: C^{l,\alpha}_\delta \to C^{l-2,\alpha}_\delta$ is a Fredholm operator
of index zero for $\delta \in (-1,n)$, see Proposition
\ref{propModelOperators}. By \cite[Lemma~3.6 (a)]{LeeFredholm} the map
$P_1: C^{l,\alpha}_\delta \to C^{l-1,\alpha}_{\delta+\tau}$ is continuous, whereas
by the Rellich Lemma, \cite[Lemma~3.6 (d)]{LeeFredholm}, the inclusion
$C^{l-1,\alpha}_{\delta+\tau} \hookrightarrow C^{l-2,\alpha}_\delta$ is compact. We
conclude that $P_1: C^{l,\alpha}_\delta \to C^{l-2,\alpha}_\delta$ is compact for
$-1<\delta<n$, and the claim follows.
\end{proof}

Recall that the constraint map $\Phi$ is defined by the formula
\eqref{eqConstraintMap2}. A direct computation shows that the
linearization of $\Phi$ is
\[ \begin{split}
D\Phi|_{(g,\pi)} (h,w)
&=
\bigg( \Delta^g (\tr^g h) -\divg^g \divg^g h + \langle h,
\ric^g \rangle \\
&\qquad
+ 2\left(1-\frac{\tr^g\pi}{n-1}\right)
\left(\tr^g w - \langle h,\pi \rangle \right) - 2 \langle h ,
\pi \circ \pi\rangle + 2 \langle \pi, w \rangle, \\
&\qquad
(\divg^g w)_k - h^{ij} \nabla_i \pi_{jk}
- (\divg h)_j \pi^j_k + \tfrac{1}{2} \nabla_j (\tr^g h) \pi^j_{k}
- \tfrac{1}{2} \pi^{ij} \nabla_k h_{ij} \bigg),
\end{split} \]
where $(\pi \circ \pi)_{ij} = g^{kl} \pi_{ik} \pi_{jl}$. The formal 
adjoint of $D\Phi$ is given by
\begin{equation} \label{eqAdjoint}
\begin{split}
D\Phi_{(g,\pi)}^*(V, X)
&=
\bigg( (\Delta V) g_{ij} - \nabla_i \nabla_j V + V \ric_{ij}
- 2 V \left(1-\frac{\tr^g\pi}{n-1}\right) \pi_{ij} \\
&\quad
-2 V \pi_{ik}\pi^k_j
+\tfrac{1}{2}(\pi_{jk} \nabla_i X^k + \pi_{ik} \nabla_j X^k)
-\tfrac{1}{2} (\divg \pi)_k X^k g_{ij} \\
&\quad
-\tfrac{1}{4} \langle \pi, \cL_X g \rangle g_{ij}
+ \tfrac{1}{2} X^k \nabla_k \pi_{ij} + \tfrac{1}{2} (\divg X) \pi_{ij}, \\
&\qquad
-\tfrac{1}{2} (\cL_X g)_{ij}
+ 2 V \left(1-\frac{\tr^g\pi}{n-1}\right) g_{ij} + 2V \pi_{ij} \bigg).
\end{split}
\end{equation}
The following lemma is the analogue of \cite[Lemma~20]{SpacetimePMT} in 
the asymptotically hyperbolic setting. The proof of the cited lemma is 
similar to \cite[Proposition~3.1]{CorvinoSchoen}.
\begin{lemma} \label{lemA}
If $(M,g,\pi)$ is an asymptotically hyperbolic initial data set of type
$(k,\alpha,\tau)$ for $k\geq 2$, $0<\alpha<1$ and $\tau>0$ then the linear map
$A: W^{2,p}_{\delta}\times W^{1,p}_{\delta}\rightarrow W^{0,p}_{\delta}$ defined
by
\[
A(h,w) = D\Phi|_{(g,\pi)} (h,w) - (0,\tfrac{1}{2} h^l_j J_l)
\]
is surjective for $1 < p < \infty$ and $-1 < \delta + \tfrac{n-1}{p} < n$.
In particular,
$D\Phi|_{(g,\pi)}:W^{2,p}_{\delta} \times W^{1,p}_{\delta}\rightarrow W^{0,p}_{\delta}$
is surjective for $1 < p < \infty$ and $-1 < \delta + \tfrac{n-1}{p} < n$.
\end{lemma}

\begin{proof} 
The first step is to show that $A$ has closed range. For this we compute
\[ \begin{split}
A(vg, \rlie_Z g) 
&=
\Big( (n-1) (\Delta^g v -n v) + (\scal^g + n(n-1))v \\ 
&\qquad
- 2 v (\tr^g \pi -\tfrac{1}{n-1}(\tr^g\pi)^2 + |\pi|^2_g) 
+ 2\langle \pi, \rlie_Z g\rangle , \\ 
&\qquad (\Delta_L Z)_i - v (\divg^g \pi )_i + 
(\tfrac{n}{2} - 1) \pi^j_i \nabla_j v
- \tfrac{1}{2} \tr^g \pi \nabla_i v 
- \tfrac{1}{2} v J_j \Big).
\end{split} \]
Reasoning as in the proof of Lemma~\ref{lemDT} we conclude that the
operator 
\[
(v,Z) \mapsto A(vg, \rlie_Z g)
\]
is a Fredholm operator $W^{2,p}_{\delta} \rightarrow W^{0,p}_{\delta}$ for 
$1 < p < \infty$ and $-1 < \delta + \tfrac{n-1}{p} < n$. Its range is
contained in the range of the operator $A$. Consequently, the range of 
the operator $A$ has finite codimension in $W^{0,p}_{\delta}$, and hence 
it is closed.
 
Next we need to show that $\ker A^*$ is trivial. Let $p^*$ be such that 
$\frac{1}{p} + \frac{1}{p^*} = 1$. Then $W^{0,p^*}_{-\delta}$ is dual to 
$W^{0,p}_{\delta}$ under the standard $L^2$ pairing, see 
\cite[Chapter~3]{LeeFredholm}. Note that we have 
$-1 < - \delta + \tfrac{n-1}{p^*} < n$ as a consequence of 
$-1 < \delta + \tfrac{n-1}{p} < n$. It follows from \eqref{eqAdjoint} 
that $\ker A^*$ consists of $(V,X) \in W^{0,p^*}_{-\delta}$ such that 
\begin{equation} \label{eqKerAdjoint}
\begin{split}
(\Delta V) g_{ij} - \nabla_i \nabla_j V + V \ric_{ij}
&=
2 V \left(1-\frac{\tr^g\pi}{n-1}\right) \pi_{ij} + 2 V \pi_{ik}\pi^k_j \\
&\qquad
-\tfrac{1}{2} (\pi_{jk} \nabla_i X^k + \pi_{ik} \nabla_j X^k ) \\
&\qquad
+ \tfrac{1}{2} (\divg \pi)_k X^k g_{ij} + \tfrac{1}{4} \langle \pi,
\cL_ X g \rangle g_{ij} \\
&\qquad
-\tfrac{1}{2} X^k \nabla_k \pi_{ij} - \tfrac{1}{2} (\divg X) \pi_{ij} \\
&\qquad
+ \tfrac{1}{4} ( X_i J_j + X_j J_i ), \\
\cL_X g
&=
4 V \left(1-\frac{\tr^g\pi}{n-1}\right) g + 4 V \pi.
\end{split}
\end{equation}
As a consequence of the second equation we have 
$\cL_X g\in W^{0,p^*}_{-\delta}$. Taking the trace of the first equation 
we have 
\begin{equation} \label{eqTrace1}
\Delta V - n V
=
V \star O^\alpha (e^{-\tau r}) + X \star O^\alpha (e^{-\tau r})
+ \cL_X g \star O^\alpha (e^{-\tau r}),
\end{equation}
where $O^\alpha (e^{-\tau r})$ denotes a section $T$ of some geometric tensor bundle of appropriate type such that $T \in C^{0,\alpha}_{\tau}$, and $A \star B$ denotes a tensor which is obtained from $A\otimes B$ by raising and lowering indices, taking a number of contractions, and switching a number of components in the product. By standard elliptic regularity \cite[Lemma~4.8 (a)]{LeeFredholm} we conclude that $V \in W^{2,p^*} _{-\delta}$. As a consequence of the second equation in
\eqref{eqKerAdjoint} we have 
\[
\rlie_X g = 4V \left( \pi - \frac{\tr^g \pi}{n} g \right).
\]
Taking the divergence, we obtain 
\begin{equation} \label{eqVectLaplacian}
\Delta_L X = 
V \star O^\alpha (e^{-\tau r}) + \nabla V \star O^\alpha (e^{-\tau r}).
\end{equation}
Thus $X \in W^{2,p^*} _{-\delta}$, again by standard elliptic regularity.
Since $(V,X) \in W^{2,p^*}_{-\delta}$ the right hand sides of equations
\eqref{eqTrace1} and \eqref{eqVectLaplacian} are both in 
$W^{0,p^*}_{-\delta+\tau}$. Using Proposition~\ref{propModelOperators}, 
improved elliptic regularity \cite[Proposition~6.5]{LeeFredholm}, 
and the continuity of the embedding 
$W^{k,p^*}_\varepsilon \hookrightarrow W^{k,p^*}_{\varepsilon'}$ for 
$\varepsilon > \varepsilon'$, we conclude that $(V,X) \in W^{2,p^*}_\gamma$ for 
any $\gamma$ such that $-1 < \gamma + \tfrac{n-1}{p^*} < n$. Therefore we 
may without loss of generality assume that 
$1 < \gamma < n - \tfrac{n-1}{p^*} = 1 + \tfrac{n-1}{p}$. 

In fact, we can show that $(V,X) \in C^{2,\beta}_\gamma$ for some $0<\beta<1$. 
Indeed, if $p^* < n$ then $(V,X) \in W^{2,\tfrac{n p^*}{n-p^*}}_\gamma$ by the 
Sobolev embedding theorem \cite[Lemma~3.6~(c)]{LeeFredholm} and standard 
elliptic regularity applied to equations \eqref{eqTrace1} and 
\eqref{eqVectLaplacian}. Repeating this argument we obtain that 
$(V,X) \in W^{2,q}_\gamma$ for some $q>n$ and thus
$(V,X) \in C^{1,\beta}_\gamma$ for some $0<\beta<1$ by Sobolev embedding
\cite[Lemma~3.6~(c)]{LeeFredholm}. Applying standard elliptic regularity
to the equations \eqref{eqTrace1} and \eqref{eqVectLaplacian} we conclude 
that $(V,X) \in C^{2,\beta}_\gamma$.

Next we show that $(V,X)$ vanishes to infinite order at infinity. That
is, $(V,X) = O(e^{-Nr})$ for any $N>0$. As a consequence of
\eqref{eqKerAdjoint} and Definition~\ref{defAHdata} we see that $(V,X)$
is a solution to the system
\begin{align*}
\hess^b V - V b
&=
V \star O(e^{-\tau r}) + \nabla V \star O(e^{-\tau r}) + X \star O (e^{-\tau r})
+ \nabla X \star O (e^{-\tau r}),\\
\cL_X b
&= 
4 V b + X \star O_1 (e^{-\tau r}) + V \star O_1 (e^{-\tau r}),
\end{align*}
where $O_1(e^{-\tau r})$ denotes a section $T$ of the appropriate geometric
tensor bundle such that $T\in C^1_\tau$. From the first equation and the
fact that $(V,X) \in C^{2,\beta}_\gamma$ we conclude that $V$ satisfies the 
ordinary differential equation
\[
\partial^2_{rr} V - V = \ftil
\]
along radial geodesic rays, where $\ftil = O (e^{-(\tau + \gamma)r})$.
Since $\tau + \gamma > 1$, it follows that $V = O (e^{-(\tau + \gamma)r})$, 
see formula \eqref{eqVarPar} for the explicit form of the solution. 
Then $V \in C^{2,\beta}_{\tau + \gamma}$ by standard elliptic regularity 
applied to \eqref{eqTrace1}. Combining this with the second equation, 
we see that $(\cL_X b)_{rr} = O(e^{-(\tau + \gamma)r})$, which yields
$\partial_r X_r = O(e^{-(\tau + \gamma)r})$. Integrating this relation from
$r$ to $\infty$, we obtain that $X_r = O(e^{-(\tau + \gamma)r})$.
Note that as a consequence of this relation we also have $\partial_\mu X_r = O(e^{-(\tau + \gamma) r})$, 
which can be seen by first differentiating with respect to $\mu$ and then integrating from $r$ to $\infty$. 
Here we work in polar
coordinates for hyperbolic space,
$(\bH^n,b) = ( (0,\infty) \times S^{n-1}, dr^2 + \sinh^2 r \, \sigma)$, 
the subscript $r$ denotes the radial component and $\mu$ denotes 
components in a coordinate system on the sphere. It follows that
\[
\partial_r X_\mu - 2 \coth r X_\mu = \bar{f},
\]
where $\bar{f} = O(e^{-(\tau + \gamma - 1)r})$,
and hence
$X_\mu = \sinh^2 r \int_r ^\infty \tfrac{\bar{f}}{\sinh^2 s} \, ds
= O(e^{-(\tau + \gamma - 1) r})$. Thus $|X|_b = O(e^{-(\tau + \gamma) r})$, and 
hence $X \in C^{2,\beta}_{\tau + \gamma}$ by standard elliptic regularity 
applied to \eqref{eqVectLaplacian}. We proceed by induction and deduce 
that $(V,X) = O(e^{-Nr})$ for any $N>0$.

To conclude the proof, note that, as a consequence of \eqref{eqKerAdjoint},
$(V,X)\in C^2$ satisfies a differential inequality
\[
|\Delta (V,X)| \leq C \left(|(V,X)| + |\nabla (V,X)| \right),
\]
where $\Delta = \nabla^* \nabla$ is the rough Laplacian. Since $(V,X)$
vanishes to infinite order at infinity, a standard unique continuation
argument, see Appendix \ref{secUC}, implies that $(V,X)$ vanishes
identically. 
\end{proof}

We use the subscript $c$ on the notation for a function space to denote the subspace of sections with compact support.

\begin{lemma} \label{lemSur}
Let $(M,g,\pi)$ be an asymptotically hyperbolic initial data set of type
$(k,\alpha,\tau)$, where $k\geq 2$, $0<\alpha<1$ and $\frac{n}{2}< \tau < n$. Then for
any $f\in C^{k-2,\alpha}_\tau$ there exist $(v,Z) \in C^{k,\alpha}_\tau$ and
symmetric 2-tensors $(h,w) \in C^{k+1,\alpha}_c$ so that
\begin{equation} \label{eqSur}
DT|_{(1,0)}(v,Z) + D\Phi|_{(g,\pi)} (h,w) = (f,\tfrac{1}{2} h^l_j J_l).
\end{equation}
If in addition $f \in C^{k-2,\alpha}_{n+\tau_0}$ for some $\tau_0>0$ then 
$(v,Z)\in C^{k,\alpha}_n$.
\end{lemma}

\begin{proof} 
For some $p > n$ we choose $\gamma > 0$ so that
$-1 < \gamma + \tfrac{n-1}{p} < \tau$. In this case
$C^{l,\alpha}_{\tau} \hookrightarrow W^{l,p}_{\gamma}$ for $l=0,1,\ldots, k$, see \cite[Lemma~3.6 (c)]{LeeFredholm}. Further, by
Lemma~\ref{lemDT} the operator $DT|_{(1,0)}:W^{2,p}_{\gamma} \to W^{0,p}_{\gamma}$
is Fredholm with index zero. Since the linear map 
$A:W^{2,p}_{\gamma} \times W^{1,p}_{\gamma} \rightarrow W^{0,p}_{\gamma}$ defined 
in Lemma~\ref{lemA} is surjective, we can find symmetric 2-tensors
$(h_i,w_i)\in W^{2,p}_{\gamma} \times W^{1,p}_{\gamma}$, $i=1,\dots, N$, such 
that their images $A(h_i,w_i)$ span a subspace that complements
$DT|_{(1,0)} (W^{2,p}_{\gamma})$ in $W^{0,p}_{\gamma}$. Note that 
by the density of compactly supported sections, 
\cite[Lemma~3.9]{LeeFredholm}, together with the continuity of $A$ we 
may assume that $(h_i,w_i) \in C^{k+1,\alpha}_c$. Consequently, since
$f\in W^{0,p}_\gamma$ we can find $(v,Z) \in W^{2,p}_{\gamma}$ and  
$(h,w) \in C^{k+1,\alpha}_c$ such that \eqref{eqSur} holds. 
By Sobolev embedding $(v,Z) \in C^{1,\alpha}_{\gamma}$. Since $\gamma > 0$ 
and $f \in C^{0,\alpha}_\tau$ it follows from \eqref{eqLinearT} that 
$(\Delta v - n v, \Delta_L Z) \in C^{0,\alpha}_{\tau}$. From 
\cite[Proposition~6.5]{LeeFredholm} we conclude that 
$(v,Z) \in C^{2,\alpha}_{\tau}$ and $(v,Z) \in C^{k,\alpha}_{\tau}$ follows by a standard bootstrap argument. 

To prove the second claim note that outside a sufficiently large compact 
set $(v,Z) \in C^{k,\alpha}_\tau$ satisfies 
$(\Delta v - n v, \Delta_L Z) \in C^{k-2,\alpha}_{n+\varepsilon}$ 
for some $\varepsilon > 0$. This is an immediate 
consequence of \eqref{eqLinearT} and the fact that $\tau > \tfrac{n}{2}$. 
The claim follows from Proposition~\ref{propCritical}.
\end{proof}

\begin{proof}[Proof of Theorem~\ref{thStrictDEC}] 
With the above lemmas at hand, the proof differs very little from that 
of \cite[Theorem~22]{SpacetimePMT}. We choose a positive $C^{k+1,\alpha}$ 
function $f$ such that 
\[
f = e^{-(n+\min\{1,\tau_0\})r}
\] 
near infinity, and let $(v,Z)\in C^{k,\alpha}_n$ and $(h,w)\in C^{k+1,\alpha}_c$
be a solution of the system 
\[
DT|_{(1,0)}(v,Z) + D\Phi|_{(g,\pi)} (h,w) = (-f,\tfrac{1}{2} h^l_j J_l),
\]
which exists by Lemma~\ref{lemSur}. We will show that for a sufficiently
small $t>0$,  
\[
\bar{g}=(1+t v)^\kappa (g + t h)
\qquad \text{and} \qquad
\bar{\pi} = (1 + tv)^{\kappa/2} (\pi + t \rlie_Z g + t w)
\]
is an initial data set whose existence is asserted in the theorem. Note
that $\|g-\gbar\|_{C^{k,\alpha}_\tau} \leq \varepsilon$,
$\|\pi-\pibar\|_{C^{k-1,\alpha}_\tau} \leq \varepsilon$ provided that $t$ is
sufficiently small.

We will verify that $\mubar > (1+\gamma) |\Jbar|_{\gbar}$ for some
$\gamma>0$ depending on $t$. Set $u = 1 + t v$ and define
\[
\Phi_1 (1 + t v, t Z, t h, t w)
= (-2 u^\kappa \mubar, u^{\kappa/2} \Jbar).
\]
Linearizing we have
\begin{equation} \label{eqLinear}
\begin{split}
\Phi_1 (1 + t v, t Z, t h, t w)
&=
\Phi_1 (1,0,0,0) + t D\Phi_1 |_{(1,0,0,0)} (v, Z, h, w) + \cR \\
&=
(-2 \mu, J) + t DT|_{(1,0)} (v,Z) + t D\Phi|_{(g,\pi)} (h,w) + \cR \\
&=
(-2 \mu, J) + t (-f, \tfrac{1}{2} h^k_i J_k) + \cR,
\end{split}
\end{equation}
where the remainder term $\cR = \cR(t,v,Z,h,w)$ can be written as
\[ \begin{split}
\cR(t,v,Z,h,w)
&=
\Phi_1 (1 + t v, t Z, t h, t w) - \Phi_1 (1,0,0,0)
- t D\Phi_1 |_{(1,0,0,0)} (v, Z, h, w) \\
&=
t \int_0^1 \left[
D\Phi_1 |_{(1 + \theta t v , \theta t Z, \theta t h, \theta t w)} - D\Phi_1 |_{(1, 0, 0, 0)}
\right] (v, Z, h, w) \, d\theta, 
\end{split} \]
by the mean value theorem.
 
We first prove that
\begin{equation} \label{cR-estimate}
|\cR| \leq C t^2 e^{-2n r} = O(t^2 f),
\end{equation}
where the constant $C>0$ does not depend on $t$ and is uniform for all
points. For this it suffices to estimate $\cR$ outside the support of
$(h,w)$ where it takes the form
\[
\cR(t,v,Z) =
t \int_0^1 \left[
DT |_{(1 + \theta t v , \theta t Z)} - DT |_{(1, 0)}
\right](v, Z) \, d\theta. 
\]
Using \eqref{eqRescConstraints} we compute 
\[ \begin{split}
DT|_{(u,Y)}(v,Z)
&=
\Big( \tfrac{4(n-1)}{n-2} (-u^{-2} v \Delta^g u + u^{-1} \Delta^g v)
-n(n-1)\kappa u^{\kappa-1}v \\
&\qquad
+ \kappa u^{\tfrac{\kappa}{2}-1} v \tr^g \pi + 2 \langle \pi,\rlie_Z g \rangle
+ 2 \langle \rlie_Z g, \rlie_Y g\rangle, \\
&\qquad
(\Delta_L Z)_j
+ \tfrac{2(n-1)}{n-2} \nabla_k (vu^{-1}) (\pi + \rlie_Y g)^k_j \\ 
&\qquad
+ \tfrac{2(n-1)}{n-2} u^{-1} \nabla_k u (\rlie_Z g)^k_j
-\tfrac{2}{n-2} \nabla_j (vu^{-1}) \tr^g \pi \Big).
\end{split} \]
Then it is not complicated to check that
\[
\left|DT |_{(1 + \theta t v , \theta t Z)} (v,Z) - DT |_{(1, 0)}(v, Z) \right|
\leq \theta t \cQ(v,Z) 
\]
where $\cQ$ is a quadratic function of $v$, its first and second order 
covariant derivatives, and $\rlie_Z g$, which is uniformly bounded in 
$\theta \in [0,1]$. Hence \eqref{cR-estimate} holds.

By \eqref{eqLinear} we have
\[
u^\kappa \mubar = \mu + \frac{t}{2} f + O(t^2 f) 
\qquad \text{and} \qquad 
u^{\kappa/2} \Jbar_i = J_i + \frac{t}{2} h^k_i J_k + O(t^2f).
\]
In particular, for sufficiently small $t>0$, we have
\begin{equation} \label{eqBoundMu}
u^\kappa \mubar > \mu + \frac{t}{3}f. 
\end{equation}
Recall that $h$ is compactly supported, hence we may write
\[
\gbar^{ij} = u^{-\kappa} (g^{ij} - t g^{ik} h^j_k + O(t^2 f)).
\]
Since $f$ is positive, we obtain
\[ \begin{split}
(u^{\kappa} |\Jbar|_{\gbar})^2
&=
u^{2\kappa} \gbar^{ij} \Jbar_i \Jbar_j \\
&=
(g^{ij} - t g^{ik} h^j_k + O(t^2 f)) (J_i + \frac{t}{2} h^l_i J_l + O(t^2f))
(J_j + \frac{t}{2} h^m_j J_m + O(t^2f)) \\
&=
|J|_g^2 + O(t^2 |J|_g f + t^3 f^2) \\
&=
\left(|J|_g + \frac{tf}{4}\right)^2 - \frac{tf}{2} |J|_g
- \frac{t^2 f^2}{16} + O\left(t \left(\frac{tf}{2} |J|_g 
+ \frac{t^2 f^2}{16}\right)\right) \\
&< \left(|J|_g + \frac{tf}{4}\right)^2
\end{split} \]
for $t>0$ small enough, so we find that
\begin{equation} \label{eqBoundJ}
u^\kappa |\Jbar|_{\gbar} < |J|_g + \frac{t}{4}f
\end{equation}
for such $t$.

Fix $t>0$ such that \eqref{eqBoundMu} and \eqref{eqBoundJ} hold. Note that
our choice of $f$ implies that $\sup_M \left( |J|_g / f\right) < \infty$.
Therefore for any $x \in M$ such that $|\Jbar|_{\gbar}(x) \neq 0$ we have 
\[
\frac{\mubar}{|\Jbar|_{\gbar}}
=
\frac{u^\kappa \mubar}{u^\kappa |\Jbar|_{\gbar}}
>
\frac{\mu + tf/3}{|J|_g + tf/4} \geq \frac{|J|_g + tf/3}{|J|_g + tf/4}
=
1 + \frac{t}{12 (|J|_g/f) + 3 t } \geq 1 + \gamma,
\]
for
\[
\gamma \definedas \frac{t}{12 \sup_M \left( |J|_g/f \right) + 3t}.
\]
At points where $|\Jbar|_{\gbar}(x) = 0$ we have $\mubar > 0$ by
\eqref{eqBoundMu}. Consequently, we have
$\mubar > (1+\gamma)|\Jbar|_{\gbar}$ everywhere on $M$ as desired. 

Note also that $(\bar{\mu},\bar{J})\in C^{k-2,\alpha}_{n+\tau'_0}$ for
$\tau_0' = \min\{1,\tau_0\}$ by \eqref{eqLinear}, the asymptotics of $u$,
and the properties of $\mathcal{R}$. 
In particular
$\|(\mu,J)-(\bar{\mu},\bar{J}) \|_ {C^{0}_{n+\tau'_0}}$ can be made arbitrarily
small for a sufficiently small $t$. Thus Proposition~\ref{propContinuity}
guarantees that 
$\left| \mathcal{M}_{(g,\pi)}(V) - \mathcal{M}_{(\bar{g},\bar{\pi})}(V) \right|
< \varepsilon$ holds.  
\end{proof}

\section{Perturbation to conformally hyperbolic asymptotics}
\label{secConfHyp}

In this section we prove the following result.

\begin{theorem} \label{thConfHyp}
Let $(M,g,\pi)$ be an asymptotically hyperbolic initial data set of type
$(k,\alpha, \tau,\tau_0)$ for $0<\alpha<1$, $\tfrac{n}{2} < \tau < n$ and
$\tau_0>0$. Assume that the dominant energy condition $\mu \geq |J|_g$
holds. Then for every $\tau' < \tau$ and $\varepsilon>0$ there exists
an asymptotically hyperbolic initial data set $(\bar{g},\bar{\pi})$, with the energy and momentum density 
denoted by $(\mubar,\Jbar)$,  which
has conformally hyperbolic asymptotics with respect to the same chart,
and is such that 
\[
\|g - \bar{g}\|_{C_{\tau'}^{k,\alpha}}<\varepsilon,
\qquad
\|\pi - \bar{\pi}\|_{C_{\tau'}^{k-1,\alpha}} < \varepsilon,
\]
the strict dominant energy condition 
\[
\bar{\mu} > |\bar{J}|_{\bar{g}}
\]
holds, and
\[
|\mathcal{M}_{(g,\pi)} (V) - \mathcal{M}_{(\bar{g},\bar{\pi})}(V)|
< \varepsilon
\]
for any $V \in \{V_{(0)}, V_{(1)}, \ldots , V_{(n)}\}$.
\end{theorem}

In \cite[Appendix B]{DGSsmallmass} a similar result was proven in the 
simpler case when $\pi=0$. The proof of Theorem~\ref{thConfHyp} is very 
similar to \cite[Proof of Theorem~18]{SpacetimePMT}. Its main ingredients 
are Theorem~\ref{thStrictDEC} and the following lemma. 

\begin{lemma} \label{lemConfHyp}
Let $(g,\pi)$ be an asymptotically hyperbolic initial data set of type
$(k,\alpha,\tau)$ for $k\geq 2$, $0<\alpha<1$ and $\tfrac{n}{2} < \tau < n$ and suppose
that $\tfrac{n}{2} < \tau'<\tau$. Then there are positive constants $C_0$ and $\delta_0$
such that for any $(\bar{\mu},\bar{J}) \in C^{k-2,\alpha}_\tau$ with
$\|(\mu-\bar{\mu}, J-\bar{J})\|_{C^{k-2,\alpha}_\tau} \leq \delta \leq \delta_0$,
there exists an initial data set $(\bar{g},\bar{\pi})$ of type $(k,\alpha,\tau)$
with the following properties:
\begin{itemize}
\item
The energy and momentum densities of $(\bar{g},\bar{\pi})$ are
$\bar{\mu}$ and $\bar{J}$.
\item
Outside of a compact set $(\bar{g},\bar{\pi})$ is of the form
\[
\bar{g} = u^\kappa b, \qquad \bar{\pi} = u^{\kappa/2} \rlie_Y b
\]
for $(u-1,Y) \in C^{k,\alpha}_\tau$.
\item
The initial data set $(\bar{g},\bar{\pi})$ is close to $(g,\pi)$ in the
sense that
\[
\|g-\bar{g}\|_{C^{k,\alpha}_{\tau'}}\leq C_0 \delta,
\qquad \|\pi-\bar{\pi}\|_{C^{k-1,\alpha}_{\tau'}}\leq C_0 \delta.
\]
\end{itemize}
\end{lemma}
\begin{proof} 
The proof uses the construction introduced by Corvino and Schoen in 
\cite[Proof of Theorem~1]{CorvinoSchoen}, which is similar to the one 
that was used in the proof of Theorem~\ref{thStrictDEC}. Given $(g,\pi)$ 
as in the statement of the theorem and $(u-1,Y) \in C^{k,\alpha}_\tau$, 
we define the map
\[
T_{(g,\pi)}(u,Y) = \Phi(u^\kappa g, u^{\kappa/2} (\pi + \rlie_Y g)).
\]
It follows from \eqref{eqRescConstraints} that the components of
$T_{(g,\pi)}$ are given by
\begin{align} \label{eqTgpi}
\begin{split}
-2 \mutil
&=
\tfrac {4(n-1)}{n-2} u^{-\kappa-1} \Delta^g u - u^{-\kappa}\scal^g - n(n-1)
+ 2 u^{-\tfrac{\kappa}{2}} \tr^g \pi \\
&\qquad
- \tfrac{1}{n-1} u^{-\kappa}(\tr^g \pi)^2
+ u^{-\kappa}\left( |\pi|_g^2 + 2 \langle \pi, \rlie_Y g \rangle
+ |\rlie_Y g|_g^2 \right), \\
\Jtil_j
&=
u^{-\tfrac{\kappa}{2}}(\Delta_L Y + \divg^g \pi)_j
+\tfrac{2(n-1)}{n-2} u^{-\tfrac{\kappa}{2}-1} (\pi + \rlie_Y g )^k_j \nabla_k u \\
&\qquad
-\tfrac{2}{n-2} u^{-\tfrac{\kappa}{2}-1} \tr^g \pi \nabla_j u,
\end{split}
\end{align}
for $j=1,2,\dots,n$. From this formula it is straightforward to compute
the linearization
\[ \begin{split}
DT_{(g,\pi)}|_{(1,0)} (v,Z)
&=
\Big( \tfrac{4(n-1)}{n-2} \Delta^g v + \tfrac{4}{n-2} \scal^g v
- \tfrac{4}{n-2}(\tr^g \pi) v \\
&\qquad
+\tfrac{\kappa}{n-1} (\tr^g\pi)^2 v- \tfrac{4}{n-2}|\pi|_g^2 v + 2\langle \pi,
\rlie_Z g \rangle, \\
&\qquad
(\Delta_L Z)_j- \tfrac{2}{n-2} (\divg^g \pi)_j v
+ \tfrac{2(n-1)}{n-2}\pi^k_j \nabla_k v
- \tfrac{2}{n-2}(\tr^g \pi) \nabla_j v \Big).
\end{split} \]
Since $\scal^g + n(n-1) \in C_\tau ^{k-2,\alpha}$, we may argue as in the
proof of Lemma~\ref{lemDT} and show that for $2\leq l \leq k$ the operator 
$DT_{(g,\pi)}|_{(1,0)}$ is Fredholm of index zero as an operator
$C^{l,\alpha}_\delta \to C^{l-2,\alpha}_\delta$ for $-1 < \delta < n$ and 
as an operator $W^{l,p}_\gamma \to W^{l-2,p}_\gamma$ for 
$-1 < \gamma + \tfrac{n-1}{p} < n$. We  can now
choose a sufficiently large $p>n$ and $\gamma $ such that $\tau' < \gamma < \tau - \tfrac{n-1}{p}$ in which case both $C^{l,\alpha}_{\tau} \hookrightarrow W^{l,p}_\gamma$ for $l=0,1,\ldots, k$ and the operator $DT_{(g,\pi)}|_{(1,0)}: W^{2,p}_\gamma \to W^{0,p}_\gamma$ is Fredholm  of index zero. Let $U$ be the subspace complementing the kernel of $DT_{(g,\pi)}|_{(1,0)}$ in $W^{2,p}_\gamma$. Arguing as in the proof of Lemma~\ref{lemSur} we conclude that there exist finitely
many pairs of compactly supported symmetric $2$-tensors $(h_i,w_i)\in C^{k+1,\alpha}_c$,
$i=1,\dots, N$, such that their images $D\Phi|_{(g,\pi)}(h_i,w_i)$ form a
basis for a subspace which complements
$DT_{(g,\pi)}|_{(1,0)}(W^{2,p}_\gamma)$ in $ W^{0,p}_\gamma$. Set
$V \definedas \operatorname{span}\{(h_i,w_i)\}_{i=1,\dots,N}$. We
define the map $\Xi_{(g,\pi)}: U \times V \to  W^{0,p}_\gamma$ by
\begin{equation} \label{eqXi}
\Xi_{(g,\pi)}:
(u,Y,h,w) \mapsto \Phi (u^\kappa g + h , u^{\kappa/2} (\pi + \rlie_Y g) + w).
\end{equation}
Then the linearization
$D\Xi_{(g,\pi)}|_{(1,0,0,0)}: U \times V \to W^{0,p}_\gamma$
is given by
\[
D\Xi_{(g,\pi)}|_{(1,0,0,0)}: 
(v,Z, \eta, \omega) \mapsto 
DT_{(g,\pi)}|_{(1,0)} (v,Z) + D\Phi|_{(g,\pi)} (\eta, \omega)
\]
and is an isomorphism by construction. 

Using the chart at infinity
$\Psi : M\setminus K_0 \to \bH^n \setminus B_{R_0}$, we define the cut-off
function $\chi_\lambda (x)= \chi(r(x)/\lambda)$, where $\chi:\bR \to \bR$
is a smooth function satisfying $\chi(r)=1$ for $r\leq 1$ and $\chi(r)=0$
for $r\geq 2$. For a sufficiently large $\lambda>0$, the cut-off initial
data $(g_\lambda ,\pi_\lambda)$ is given by
\[ 
g_\lambda = \chi_\lambda g + (1-\chi_\lambda)\Psi^* b,
\qquad
\pi_\lambda = \chi_\lambda \pi.
\] 
Now for any $\tau_1 < \tau$ we have
\begin{equation} \label{eqCloseData}
\|g - g_\lambda\|_{C^{k,\alpha}_{\tau_1}}\to 0
\qquad \text{and} \qquad
\|\pi - \pi_\lambda \|_{C^{k-1,\alpha}_{\tau_1}}\to 0
\end{equation}
as $\lambda\to \infty$. Hence we also have
\begin{equation} \label{eqCloseConstr}
\|\Phi(g,\pi)- \Phi(g_\lambda,\pi_\lambda)\|_{C^{k-2,\alpha}_{\tau_1}} \to 0
\end{equation}
as $\lambda\to \infty$.

Similarly to \eqref{eqXi}, we define the map
$\Xi_{(g_\lambda, \pi_\lambda)}:U \times V \to W^{0,p}_\gamma$ by
\[
\Xi_{(g_\lambda,\pi_\lambda)}: (u,Y,h,w) \mapsto
\Phi (u^\kappa g_\lambda + h , u^{\kappa/2} (\pi_\lambda + \rlie_Y g_\lambda) + w).
\]
The linearization $D\Xi_{(g_\lambda,\pi_\lambda)}|_{(1,0,0,0)}:
U \times V \to  W^{0,p}_\gamma$ is given by
\[
D\Xi_{(g_\lambda,\pi_\lambda)}|_{(1,0,0,0)}:
(v,Z, \eta, \omega) \mapsto
DT_{(g_\lambda,\pi_\lambda)}|_{(1,0)} (v,Z) 
+ D\Phi|_{(g_\lambda,\pi_\lambda)} (\eta, \omega).
\]
As a consequence of \eqref{eqCloseData}, the operators
$D\Xi_{(g_\lambda,\pi_\lambda)}|_{(1,0,0,0)}$ converge to the isomorphism
$D\Xi_{(g,\pi)}|_{(1,0,0,0)}$ as $\lambda\to \infty$ in the uniform operator
topology. It follows that there exists a positive $\lambda_0$ such that for
any $\lambda \geq \lambda_0$ the linearization
$D\Xi_{(g_\lambda,\pi_\lambda)}|_{(1,0,0,0)}$ is an isomorphism. Note that
$\Xi_{(g_\lambda,\pi_\lambda)} (1,0,0,0) = \Phi(g_\lambda, \pi_\lambda)
= (-2\mu_\lambda, J_\lambda)$. Applying the Inverse Function Theorem, see 
for example \cite[Theorem~4.2 and Remark~4.3]{InverseFunction}), it is
not complicated to check that there exists $\rho_0>0$ depending only
on $(g,\pi)$ such that $\Xi_{(g_\lambda,\pi_\lambda)}: B_{\rho_0} (1,0,0,0) \to
\Xi_{(g_\lambda,\pi_\lambda)} \left(B_{\rho_0} (1,0,0,0)\right)$ is a diffeomorphism
for any $\lambda \geq \lambda_0$. Furthermore, there exists a constant
$C>0$ depending only on $(g,\pi)$ such that 
\begin{equation} \label{eqContInverse}
C \|(u,Y,h,w) - (1,0,0,0)\|_{W^{2,p}_{\gamma}}
\leq
\| 
\Xi_{(g_\lambda,\pi_\lambda)}(u,Y,h,w) - (-2\mu_\lambda, J_\lambda) 
\|_{W^{0,p}_{\gamma}}
\end{equation}
holds for any $(u, Y, h, w) \in B_{\rho_0} (1,0,0,0)$, and such that
\[
B_{C \rho_0} (-2\mu_\lambda, J_\lambda) \subset
\Xi_{(g_\lambda,\pi_\lambda)} \left(B_{\rho_0} (1,0,0,0)\right).
\]

Now suppose that $(\mubar,\Jbar)$
is such that $\|(2(\mu-\bar{\mu}), J-\bar{J})\|_{C^{0,\alpha}_\tau} \leq
\delta$. By \eqref{eqCloseConstr} we may assume that
$\lambda \geq \lambda_0$ is such that
$\|(2(\mu-\mu_\lambda), J-J_\lambda)\|_{C^{0,\alpha}_{\tau_1}} \leq \delta$ where $\tau_1 < \tau$. If we further assume that $\tau_1 > \gamma + \frac{n-1}{p}$ so that $C^{0,\alpha}_{\tau_1} \hookrightarrow W^{0,p}_\gamma$ it follows that there exists $\delta_0>0$ depending on $C$ and $\rho_0$ such that 
$(-2\bar{\mu},\bar{J}) \in B_{C \rho_0} (-2\mu_\lambda, J_\lambda)$ as long as $\delta \leq \delta_0$. Then it
follows from the above discussion that there exists
$(u,Y,h,w) \in B_{\rho_0} (1,0,0,0)$ such that
$\Xi_{(g_\lambda,\pi_\lambda)}(u,Y,h,w) = (-2\bar{\mu},\bar{J})$. As a consequence
of \eqref{eqContInverse} 
\[
\|(u-1,Y,h,w)\|_{W^{2,p}_{\gamma}} \leq C_0 \delta
\]
for a uniform constant $C_0$. By the Sobolev embedding we have $(u-1,Y) \in C^{1,\alpha}_{\gamma}$ and it follows from \eqref{eqTgpi} that outside of a compact set
$(u,Y)$ satisfies 
\begin{equation} \label{eqConfIteration}
\begin{split}
-2u^{\kappa+1} \bar{\mu}
&=
\tfrac{4(n-1)}{n-2}\Delta^b u + n(n-1) (u-u^{\kappa+1}) + u |\rlie_Y b|_b^2, \\
u^{\kappa/2} \bar{J}_j
&=
(\Delta_L Y)_j + \tfrac{2(n-1)}{n-2} u^{-1} (\rlie_Y b)^k_j \nabla_k u.
\end{split}
\end{equation}
Consequently, $v=u-1$ and $Y$ are such that
$(v,Y) \in C^{1,\alpha}_{\gamma}$ and recalling that $\gamma>n/2$ it follows from \eqref{eqConfIteration} that
\[
(\Delta^b v - n v, \Delta_L Y ) \in C^{0,\alpha}_{\tau}.
\]
Hence $(v,Y) \in C^{2,\alpha}_\tau$ by the improved elliptic regularity
\cite[Proposition~6.5]{LeeFredholm}. Further regularity follows by a standard bootstrap argument.

It is now straightforward to check that there exists a constant $C_0 > 0$
such that the initial data
\[
\bar{g} = u^\kappa g_\lambda + h, \qquad
\bar{\pi} = u^{\kappa/2} (\pi_\lambda + \rlie_Y g_\lambda) + w
\]
has all the required properties. In particular, the last claim of the theorem follows using the fact that $\|(u-1,Y)\| _{ C^{1,\alpha}_{\tau'}} \leq C_0 \delta$ (up to increasing $C_0$ if necessary) as a consequence of $\gamma > \tau'$ and the Sobolev embedding, and applying Schauder estimates \cite[Lemma~4.8(b)]{LeeFredholm} throughout.
\end{proof}

The following result is not complicated to prove.

\begin{proposition} \label{propConfHyp}
Suppose that $(M,g,\pi)$ is an asymptotically hyperbolic initial data
set of type $(k,\alpha, \tau, \tau_0)$ for $k\geq 2$, $0 < \alpha < 1$ and $\tau_0 \geq 1$ such that 
\[
g=u^\kappa b, \qquad \pi = u^{\kappa/2} \rlie_Y b
\]
outside of a compact set for some $(u-1,Y)\in C^{k,\alpha}_{\tau}$. Then
$(g,\pi)$ has conformally hyperbolic asymptotics in the sense of
Definition~\ref{defConfHypID}. 
\end{proposition}
\begin{proof}
Let $v=u-1$. It follows from \eqref{eqConfIteration} that outside of
a compact set
\[
(\Delta^b v - n v, \Delta_L Y ) \in C^{k-2,\alpha}_{n+\varepsilon}
\]
for some $\varepsilon > 0$. Then $(v,Y) \in C^{k,\alpha} _n$ by Proposition
\ref{propCritical}. More specifically, from the proof of Proposition
\ref{propCritical} we see that $(v,Y)$ is of the form \eqref{eqConfHyp},
where $(v_0, Y_0)$ does not depend on $r$, and
$(v_1,Y_1) \in C^{2,\alpha}_{k+\varepsilon}$. 

Inserting $u=v+1$ and $Y$ in \eqref{eqConfIteration} we conclude 
that $(v_1,Y_1) \in C^{k,\alpha}_{n+\varepsilon}$ satisfies 
\[
(\Delta^b v_1 - n v_1, \Delta_L Y_1 ) \in C^{k-2,\alpha}_{n+1}.
\]
Thus $(v_1,Y_1) \in C^{k,\alpha} _{n+1}$ by Proposition~\ref{propCritical}.
\end{proof}

\begin{proof}[Proof of Theorem~\ref{thConfHyp}]
By Theorem~\ref{thStrictDEC} we may without loss of generality assume
that $\mu > (1 + \gamma) |J|_g$ for some $\gamma > 0$. Let $\xi$ be a 
smooth function such that $\xi(x) = e^{-r(x)}$ outside a compact set. For 
$\chi_\lambda$ as in the proof of Lemma~\ref{lemConfHyp} set 
$\xi_\lambda \definedas \chi_\lambda + (1 - \chi_\lambda) \xi$. Then 
$ (\xi_\lambda \mu,\xi_\lambda J) \in C^{k-2,\alpha} _{n + 1 + \tau_0}$ and
\begin{equation} \label{eqConstrConv}
\|(\mu,J) - (\xi_\lambda \mu, \xi_\lambda J)\|_{C^{k-2,\alpha}_{n+\tau_0'}} \to 0
\end{equation}
as $\lambda \to \infty$ for any $\tau_0'<\tau_0$. By Lemma~\ref{lemConfHyp} 
and Proposition~\ref{propConfHyp} we may construct initial data sets
$(g_\lambda,\pi_\lambda)$ with conformally hyperbolic asymptotics, whose 
energy and momentum densities are $(\xi_\lambda \mu, \xi_\lambda J)$, and 
such that
\begin{equation} \label{eqDataConv}
\|g - g_\lambda\|_{C^{k,\alpha}_{\tau'}} \to 0, \qquad 
\|\pi - \pi_\lambda\|_{C^{k-1,\alpha}_{\tau'}} \to 0
\end{equation}
for any $\tau'<\tau$ as $\lambda \to \infty$. In particular, 
$\|g-g_\lambda\|_{C^0} \to 0$ as $\lambda \to \infty$ thus 
$|J|^2_{g_\lambda} = |J|_g^2 (1 + o(1))$ as $\lambda \to \infty$. 
It follows that 
\[
\xi_\lambda \mu > \xi_\lambda (1+\gamma) |J|_g
\geq
\left(1+\frac{\gamma}{2}\right) |\xi_\lambda J|_{g_\lambda} 
\]
for $\lambda$ sufficiently large, hence $(g_\lambda,\pi_\lambda)$ satisfies 
the strict dominant energy condition. Further, since \eqref{eqConstrConv} 
and \eqref{eqDataConv} hold for any 
$0 < \tau'_0 < \tau_0$ and $\tfrac{n}{2} < \tau' <\tau$, it follows by
Proposition~\ref{propContinuity} that
\[
\mathcal{M}_{(g_\lambda,\pi_\lambda)}(V) \to \mathcal{M}_{(g,\pi)} (V)
\] 
for all $V\in \{V_{(0)}, V_{(1)}, \ldots, V_{(n)}\}$ as $\lambda \to \infty$.
\end{proof}

\section{Initial data sets with Wang's asymptotics}
\label{secWang}

In this section we refine the results of Section~\ref{secStrictDEC} for another important class of asymptotically hyperbolic initial data.

\begin{definition} \label{defWangData}
Let $(M,g,\pi)$ be an asymptotically hyperbolic initial data set of type
$(k,\alpha,n,\tau_0)$ for $k\geq 2$, $0 \leq \alpha < 1$ and $\tau_0>0$. We say that $(M,g,\pi)$ has
{\em Wang's asymptotics} with respect to the chart at infinity
\[
\Psi: M \setminus K_0 \to \bH^n \setminus \overline{B}_{R_0}
\]
if the following holds:
\begin{enumerate}
\item
The pushforward of the metric $g$ under $\Psi$ satisfies
\begin{equation} \label{eqWangMetric}
\Psi_* g = dr^2 + \sinh^2 r g_r ,
\end{equation}
where 
\[
g_r = \sigma + m e^{-nr} + O^{k,\alpha}(e^{-(n+1)r})
\]
is an $r$-dependent family of symmetric $2$-tensors on $S^{n-1}$,
$m \in C^{k,\alpha}$ is a symmetric 2-tensor on $S^{n-1}$, and the expression $O^{k,\alpha}(e^{-(n+1)r})$ stands for a tensor in the weighted H\"older space $C^{k,\alpha}_{n+1}(\bH^n)$.
\item
The pushforward of the $2$-tensor $\pi$ under $\Psi$ satisfies
\begin{equation}\label{eqWangPi}
\begin{split}
(\Psi_* \pi)_{rr} & = p_{rr} e^{-nr} + q_{rr}, \\
(\Psi_* \pi)_{r\mu} & = p_{r\mu} e^{-(n-1)r} + q_{r\mu}, \\
(\Psi_* \pi)_{\mu\nu} & = p_{\mu\nu} e^{-(n-2) r} + q_{\mu\nu},
\end{split}
\end{equation}
where $p\in C^{k-1,\alpha}_{loc}$ does not depend on $r$,  $q\in C^{k-1,\alpha}_{n+1}$, and $\mu, \nu$ denote components in a
coordinate system on the sphere.
\end{enumerate}
\end{definition}

Asymptotically hyperbolic metrics $g$ with asymptotics
\eqref{eqWangMetric} were considered by Wang in \cite{WangMass} and
have been studied in various contexts, see for example
\cite{AnderssonCaiGalloway}, \cite{BalehowskyWoolgar}, \cite{Neves}, and \cite{NevesTian}.

Note that any sufficiently regular conformally compactifiable metric with the round
sphere as the boundary at infinity and deviating from the
hyperbolic metric at the ``critical'' order $|g-b|_b=O(e^{-nr})$
can be written in the form \eqref{eqWangMetric} in appropriate coordinates.
See, for example, \cite[Section~IV]{BalehowskyWoolgar}, 
\cite[Section~3]{AnderssonCaiGalloway}, or \cite{massaspect}. We will make use of this fact in the proofs of Theorem~\ref{thPerturbationWang1} and Theorem~\ref{thPerturbationWang2} below.

In the case when $(M,g,\pi)$ is an initial data set with Wang's
asymptotics a direct computation shows that the mass functional is
given by
\[
\mathcal{M} (V_{(0)}) =
\tfrac{1}{2(n-1)\omega_{n-1}}
\int_{S^{n-1}} (n \tr_\sigma m - 2 p_{rr}) \, d\mu^\sigma,
\]
and
\[
\mathcal{M} (V_{(i)}) = \tfrac{1}{2(n-1)\omega_{n-1}}
\int_{S^{n-1}} x^i (n \tr_\sigma m - 2 p_{rr}) \, d\mu^\sigma,
\]
for $i=1, \dots, n$, where $m$ and $p_{rr}$ are as in Definition
\ref{defWangData}. Furthermore, we have the following result.

\begin{theorem}\label{thPerturbationWang1}
Let $(M,g,\pi)$ be an asymptotically hyperbolic initial data set of type $(k+1,\alpha,n,\tau_0)$ for $k\geq 2$, $0 < \alpha < 1$ and $\tau_0>0$ such that it has Wang's asymptotics with respect to the chart at infinity $\Psi: M\setminus K_0 \to \mathbb{H}^n \setminus \overline{B}_{R_0}$ and satisfies
the dominant energy condition $\mu\geq |J|_g$. Then, for any
$\varepsilon>0$ there exists an asymptotically hyperbolic initial data
set $(\bar{g}, \bar{\pi})$, with the energy and momentum density 
denoted by $(\mubar,\Jbar)$,  of type  $(k+1,\alpha,n,\tau'_0)$ for some $\tau'_0>0$ such
that
\begin{equation} \label{eqClose}
\|g-\bar{g}\|_{C^{k+1,\alpha}_{n}} < \varepsilon,
\qquad \text{and} \qquad 
\|\pi-\bar{\pi}\|_{C^{k,\alpha}_n} < \varepsilon,
\end{equation} 
the strict dominant energy condition 
\begin{equation} \label{eqStrictDec2}
\bar{\mu} > |\bar{J}|_{\bar{g}}
\end{equation} 
holds, and
\begin{equation} \label{eqMassChange}
|\mathcal{M}_{(g,\pi)}(V) - \mathcal{M}_{(\bar{g},\bar{\pi})}(V)|
< \varepsilon
\end{equation}
for any $V \in \{V_{(0)}, V_{(1)}, \ldots , V_{(n)}\}$. Furthermore, there is a coordinate chart at infinity $\Phi: M\setminus K_0 \to \mathbb{H}^n \setminus \overline{B}_{R_0}$  such that $(M,\bar{g}, \bar{\pi})$ is an asymptotically hyperbolic initial data set of type $(k,\alpha,n,\tau'_0)$ with Wang's asymptotics with respect to this chart.
\end{theorem}

\begin{proof}
Let $\varepsilon>0$ be fixed. Since
$C^{k,\alpha}_n \hookrightarrow C^{k,\alpha}_\tau$ for $n > \tau$ and
$k=0,1,\dots$, we may view $(M,g,\pi)$ as initial data of type $(k+1,\alpha, \tau, \tau_0)$ for $k\geq 2$, $0<\alpha <1$, $\frac{n}{2}<\tau < n$ and $\tau_0 > 0$. Arguing as in the proof of Theorem~\ref{thStrictDEC} one 
shows that there exist $(v,Z)\in C^{k,\alpha}_n$, and $(h,w)\in C^{k+1,\alpha}_c$ such that for
some sufficiently small $t>0$ the perturbed initial data set
\[
\bar{g}=(1+t v)^\kappa (g + t h)
\qquad \text{and} \qquad
\bar{\pi} = (1 + tv)^{\kappa/2} (\pi + t \rlie_Z g + t w)
\]
satisfies \eqref{eqClose} and \eqref{eqStrictDec2}. Moreover, if the positive function $f$ used in this construction is chosen so $f = O(e^{-(n+1)r})$ then we have $(v,Z) = (v_0,Z_0) e^{-nr} + (v_1,Z_1)$ for $(v_0,Z_0)\in C^{k,\alpha}_{loc}$
independent of $r$ and $(v_1,Z_1) \in C^{k,\alpha}_{n+1}$ as a consequence of \eqref{eqLinearT} and the fact that $(M,g,\pi)$ is initial data of type $(k+1,\alpha, n, \tau_0)$ (compare the proof
of Proposition~\ref{propConfHyp}). Since in this case $f$ might decay
faster than $J = O(e^{-(n+\tau_0) r})$, it is not clear that there is a $\gamma>0$
such that $\bar{\mu} > (1+\gamma) |\bar{J}|_{\bar{g}}$. However, this is
not important in the current setting since we do not intend to make a
further perturbation of $(\bar{g},\bar{\pi})$. 

Next we estimate the difference between the masses of the initial data sets $(g,\pi)$ and $(\bar{g},\bar{\pi})$. Outside a compact set we have 
\begin{equation}\label{eqBarredID}
\bar{g}=(1+t v)^\kappa g
\qquad \text{and} \qquad
\bar{\pi} = (1 + tv)^{\kappa/2} (\pi + t \rlie_Z g).
\end{equation}
Set $U = (1 + t v)^\kappa - 1$ , then $\bar{g} - g = U g$, where $U = O_1(tv) = O_1(e^{-nr})$. We also have 
\begin{equation*}
\bar{\pi} - \pi =  \left((1 + t v)^{\kappa/2}- 1 \right) \pi + t (1 + tv)^{\kappa/2} \rlie_Z g = t \rlie_Z b + O(e^{-2nr}).
\end{equation*}
Then a straightforward computation shows that
for any $V\in \{V_{(0)}, V_{(1)}, \ldots , V_{(n)}\}$ we have 
\[ \begin{split}
&
\mathcal{M}_{(\bar{g},\bar{\pi})}(V) - \mathcal{M}_{(g,\pi)}(V)\\ 
&\quad=
\tfrac{1}{2(n-1)\omega_{n-1}} \lim_{R \to \infty}
\int_{S_R} \left((n-1) (U d V - V d U)(\nu)
-2t \rlie_Z b (\nabla^b V,\nu)\right)\, d\mu^b.
\end{split} \]
Consequently, we have
\[
|\mathcal{M}_{(\bar{g},\bar{\pi})}(V) - \mathcal{M}_{(g,\pi)}(V)|
\leq Ct (|v|_{C^1_n} + |Z|_{C^1_n})
\]
for any $V\in \{V_{(0)}, V_{(1)}, \ldots , V_{(n)}\}$ and
\eqref{eqMassChange} follows, after decreasing $t$ if necessary. 

Now recall that $g$ has asymptotic expansion \eqref{eqWangMetric}
with respect to the chart $\Psi$ at infinity. With respect to this chart,
$\bar{g}$ has the expansion
\[
\bar{g} =
(1 + t v)^\kappa dr^2
+ \sinh^2 r
\left(\sigma + ( m + t \kappa v_0 \sigma) e^{-n r} + O^{k+1,\alpha}(e^{-(n+1)r}) \right),
\]
where the term $O^{k+1,\alpha}(e^{-(n+1)r})$ is an $r$-dependent tensor on $S^{n-1}$ as described in Definition \ref{defWangData}. While this expansion is not of the form \eqref{eqWangMetric}, there are standard techniques to find coordinates near infinity in which $\bar{g}$ has the desired form
(see for example \cite[Section 3.2.1]{AnderssonCaiGalloway}, \cite[Section~IV]{BalehowskyWoolgar}, \cite{CGNP}, or \cite[Section~2.2]{massaspect}). 
For completeness, we provide the details here.

First, using the substitution 
\[
r = \arcsinh\left(\sinh^{-1} \rho\right)
\]
we bring $\bar{g}$ to the conformally compact form
\[
\bar{g} = \sinh^{-2} \rho \left((1 + t v)^\kappa d\rho^2 + \sigma + 2^{-n}( m + t \kappa v_0 \sigma)\rho^n  + \eta \right).
\]
The expression in the brackets is a $C^{k+1,\alpha}$-metric on $\{0\leq \rho \leq \rho_0\}$ for some $\rho_0 > 0$, and $\eta=\eta_{\mu\nu} dy^\mu dy^\nu$ is a $\rho$-dependent tensor on $S^{n-1}$ with components $\eta_{\mu\nu}=O^{k+1,\alpha}(\rho^{n+1})$ which is to be understood in the following sense:
\begin{equation}\label{eqBoundaryFallOff}
\begin{split}
f &= O^{K,\alpha} (\rho^N) \Leftrightarrow \\ 
\rho^{l-N} \partial^l_\rho \partial^{(m)} _{y^i} f &\in C^{K-l-|m|,\alpha} (\{0\leq \tau \leq \tau_0\}) \text{ for } 0\leq l+|m| \leq K.
\end{split}
\end{equation}
Note that $\rho$ is a smooth defining function in the sense of \cite{Lee-spectrum} and the manifold has the round sphere $(S^{n-1},\sigma)$ as conformal infinity.  
 
Next, we eliminate the term $2^{-n} t \kappa v_0 \rho^n$ in the coefficient 
\[
(1 + t v)^\kappa = 1 + 2^{-n} t \kappa v_0 \rho^n + O^{k+1,\alpha}(\rho^{n+1})
\]
by changing the defining function according to 
\[
\rho = \tau - \frac{t \kappa}{n 2^{n+1}} v_0\, \tau^{n+1}.
\]
This gives us
\begin{equation}\label{eqMetricInTau}
\begin{split}
\bar{g} = & \sinh^{-2}\tau \left\{\left(1+ O^{k+1,\alpha}(\tau^{n+1})\right)d\tau^2 - \frac{t\kappa}{n 2^n} \tau^{n+1}\partial_\mu v_0 \, d\tau dy^\mu\right. \\
& \quad +  \left[ \sigma_{\mu\nu} + 2^{-n}\left(m_{\mu\nu} + \tfrac{t \kappa (n+1)}{n } v_0 \sigma_{\mu\nu}\right) \tau^n \right.\\ & \qquad +  \left.\left.\frac{t^2 \kappa^2}{n^2 4^{n+1}} \partial_\mu v_0 \partial_\nu v_0 \tau^{2n+2} + O^{k+1,\alpha}(\tau^{n+1})\right] dy^\mu dy^\nu\right\}
\end{split}
\end{equation}
where the notation $O^{K,\alpha}(\tau^N)$ is understood as in \eqref{eqBoundaryFallOff} with $\rho$ replaced by $\tau$. 

We will now perform the change of conformal gauge as described in \cite[Section 3.2.1]{AnderssonCaiGalloway}. The idea is to write 
\[
\gbar \definedas (\sinh \tau)^{-2} \gtil= (\theta \sinh \tau)^{-2} (\theta^2 \gtil)
\]
and then choose the function $\theta$ so that $\theta \sinh \tau = \sinh \chi$, where $\chi$ is the geodesic distance to the boundary $\{\tau = 0\}$ with respect to the metric $\theta^2 \gtil$. For this $\theta$ is required to satisfy the equation
\begin{equation}\label{eqODEnonlin}
\phi \gtil(d\theta, d\theta) + 2 \theta \gtil(d\theta, d\phi) = \theta^4 \phi + \theta^2 \phi^{-1} (1 - \gtil (d\phi, d\phi)),
\end{equation}
where $\phi = \sinh \tau$, see \cite[Section 3.2.1]{AnderssonCaiGalloway}. This is a first order PDE with characteristics transversal to the boundary $\{\tau = 0\}$ so the solution $\theta$ satisfying the boundary condition $\theta = 1$ exists in $\{0 \leq \tau \leq \tau_0\}$ for some $\tau_0>0$. Due to the regularity of the coefficients of \eqref{eqODEnonlin} we conclude that $\theta \in C^{k,\alpha}(\{0 \leq \tau \leq \tau_0\})$\footnote{In particular, we only have $\phi^{-1} (1 - \gtil (d\phi, d\phi)) \in C^{k,\alpha}(\{0 \leq \tau \leq \tau_0\})$, see also \eqref{eqMetricInTau}. To establish the regularity of the solution, one may argue as in the proof of \cite[Theorem 22.39]{LeeBook} while keeping track of regularity as it is done in the proof of \cite[Lemma 5.1]{Lee-spectrum}. The details are left to the reader.}, that is at this step there is a loss of regularity by one derivative. 
Similar to \cite[Section 3.2.1]{AnderssonCaiGalloway} we conclude that $\theta = 1 + O^{k,\alpha}(\tau^{n+1})$ and $\chi = \tau(1 + O^{k,\alpha}(\tau^{n+1}))$. All in all, we obtain
\[
\bar{g} = \sinh^{-2} \chi \left(d\chi^2 + \sigma + 2^{-n}\left(m + \tfrac{t \kappa (n+1)}{n } v_0 \sigma\right) \chi^n  + \bar{\eta}\right),
\]
where $\bar{\eta}=\bar{\eta}_{\mu\nu} dy^\mu dy^\nu$ is a $\chi$-dependent tensor on $S^{n-1}$ with components $\eta_{\mu\nu}=O^{k,\alpha}(\chi^{n+1})$ in the sense of \eqref{eqBoundaryFallOff} with $\rho$ replaced by $\chi$.

We conclude by performing the coordinate change $\bar{r}=\arcsinh(\sinh^{-1}\chi)$ and obtain 
\[
\bar{g}
= d\bar{r}^2 + \sinh^2 \bar{r}
\left( \sigma + (m + \tfrac{t\kappa(n+1)}{n} v_0 \sigma) e^{-n \bar{r}}
+ O^{k,\alpha}(e^{-(n+1)\bar{r}}) \right),
\]
with the $O^{k,\alpha}(e^{-(n+1)\bar{r}})$  as in Definition \ref{defWangData}, which is in the form \eqref{eqWangMetric}.
Finally, we note that the coordinate change $r\to\bar{r} = r - \frac{t\kappa}{2n} v_0 e^{-n r} + O^{k,\alpha}(e^{-(n+1) r})$ does not change the mass of the initial data set $(\bar{g},\bar{\pi})$ which is readily checked by a direct computation. In fact, the effect of this change can be described in rough terms as moving the mass content of the metric $\bar{g}$ from the radial part $\bar{g}_{rr}$ to the tangential part $\bar{g}_{\mu\nu}$, while preserving the mass. 
\end{proof}

The change of the radial coordinate performed in the proof of Theorem~\ref{thPerturbationWang1} gives us a mean to modify conformally hyperbolic asymptotics to Wang's asymptotics. In particular, it is straightforward to obtain the following consequence of Theorem~\ref{thConfHyp}.

\begin{theorem}\label{thPerturbationWang2}
Let $(M,g,\pi)$ be an asymptotically hyperbolic initial data set of type
$(k+1, \alpha, \tau, \tau_0)$ for $0<\alpha<1$, $\tfrac{n}{2} < \tau < n$ and
$\tau_0>0$. Assume that the dominant energy condition $\mu \geq |J|_g$
holds. Then for every $\varepsilon>0$ there exists
an asymptotically hyperbolic initial data set $(\bar{g},\bar{\pi})$ of type
$(k, \alpha, n , \tau'_0)$ for some $\tau'_0>0$ with Wang's asymptotics (possibly with respect to a different chart at infinity) satisfying 
the strict dominant energy condition 
\[
\bar{\mu} > |\bar{J}|_{\bar{g}}
\]
and such that 
\[
|\mathcal{M}_{(g,\pi)} (V) - \mathcal{M}_{(\bar{g},\bar{\pi})}(V)|
< \varepsilon
\]
for any $V \in \{V_{(0)}, V_{(1)}, \ldots , V_{(n)}\}$.
\end{theorem}

\begin{proof}
By Theorem~\ref{thConfHyp}, we can approximate $(g,\pi)$ by an initial data set $(\bar{g},\bar{\pi})$ of the same regularity and with conformally hyperbolic asymptotics. Noting that outside a compact set $(\bar{g},\bar{\pi})$ is of the form \eqref{eqBarredID} with $(g,\pi)=(b,0)$ the result follows by performing the coordinate change $r\to \bar{r}$ as in the proof of Theorem~\ref{thPerturbationWang1}.
\end{proof}

\section{Concluding remarks}
\label{secRemarks}

\begin{remark}
In this paper we have focused on the charge integrals
$\mathbb{Q}_{(V, -dV)}$, where $V\in \{V_{(0)}, V_{(1)},\ldots, V_{(n)}\}$.
These charge integrals are associated (as described in
Section~\ref{secCharges}) to the Killing vectors
$\partial_t, \partial_{x^1}, \ldots, \partial_{x^n}$ of Minkowski spacetime
which generate infinitesimal translations in time and space. One may ask
if the analogue of Theorem~\ref{thConfHyp} can be proven for the remaining
charges associated with the Killing vectors
$x^i \partial_t + t \partial_{x^i}$, $1 \leq i \leq n$, and
$x^i \partial_{x^j} - x^j \partial_{x^i}$, $1\leq i < j \leq n$, which
generate respectively infinitesimal boosts and rotations. In fact, using
the general theory by Michel \cite[Section~IV.B]{MichelMass} it is
straightforward to check that these charges are well-defined and
continuous under the assumptions of Proposition~\ref{propWellDefMass}
and Proposition~\ref{propContinuity}. As a consequence, we expect that
the perturbation results of this paper can be extended to apply to these
charges as well. Note that this is quite different from the situation
in the asymptotically Euclidean setting, where the charges associated with
boosts and rotations are determined by terms of lower order in the asymptotic expansion of the initial data set than the charges associated with translations in time and space. 
For this reason a given asymptotically Euclidean initial data set can be
perturbed slightly to achieve any value of angular momentum and center
of mass in such a way that the mass and linear momentum do not change,
see Huang, Schoen, and Wang \cite{HuangSchoenWang}. (In particular, this
shows that the mass and angular momentum inequality will in general not
hold for asymptotically Euclidean initial data sets without the assumption
of axial symmetry.) One does not expect such a result to hold in the
asymptotically hyperbolic setting. 
\end{remark}

\begin{remark}
Using the appropriate notion of mass (see for example \cite[Section 4.2]{MichelMass} or \cite[Section 4]{CoM}) it is straightforward to extend our results to the case of asymptotically hyperbolic initial data representing slices of asymptotically anti-de Sitter spacetimes. 
\end{remark}

\begin{remark}
Regarding the extension of the results to the case of weighted Sobolev
spaces, note that in this case it is not possible to rely on the
beautiful work of Lee \cite{LeeFredholm}. Instead, methods for
operators \emph{asymptotic to geometric operators on hyperbolic space}
(compare Bartnik \cite[Definition~1.5]{Bartnik}) can be used, see for
example the proof of Lemma~\ref{lemDT}. Some results in this direction 
have been obtained in \cite{DelayFougeirol}.
\end{remark}

\appendix
\section{Fredholm operators on asymptotically hyperbolic manifolds: chart-dependent approach}
\label{appFredholm}

Theorem~C in Lee's monograph \cite{LeeFredholm} proves the Fredholm
property of geometric elliptic operators acting on weighted Sobolev
and H\"older spaces on conformally compact manifolds. In this appendix
we will show that the same result holds for asymptotically hyperbolic
manifolds in the sense of Definition~\ref{defAHmanifold}. We use the
same definition of geometric tensor bundles and geometric elliptic partial
differential operators as in the cited monograph.

Let $(M,g)$ be a $C^{l,\beta}_\tau$-asymptotically hyperbolic $n$-manifold
in the sense of Definition~\ref{defAHmanifold} for $n\geq 2$, $l\geq 2$,
$0\leq \beta< 1$, and $\tau>0$. Let
$\Psi: M\setminus K_0 \to \bH^n \setminus \overline{B}_{R_0}$ be the chart
at infinity. Given a geometric elliptic partial differential operator
$P: C^\infty (M;E) \to C^\infty (M;E)$ of order $m \leq l$ we define the
{\em indicial map} $I_s(P): E|_{S^{n-1}} \to E|_{S^{n-1}}$ by setting
\[
I_s(P)\overline{u} \definedas \lim_{r\to \infty} e^{sr} P(e^{-sr} \overline{u}).
\]
Following \cite[Section~4]{LeeFredholm}, we call $s \in \mathbb{C}$ a
{\em characteristic exponent} at $p\in S^{n-1}$ if $I_s(P)$ is singular at
$p$. Using the fact that $|\Psi_* g - b |_b = O(e^{-\tau r})$ it is
not complicated to check that the characteristic exponents of $P$ are
constant on $S^{n-1}$. Further, if $P$ is formally self-adjoint one may verify
that the set of characteristic exponents is symmetric about the line
$\operatorname{Re} s = \tfrac{n-1}{2} - k$, where $k=k_1 - k_2$ is the
rank of the geometric tensor bundle $E \subset T^{k_1}_{k_2} M$, see
\cite[Proposition~4.4]{LeeFredholm}. Similarly to the conformally
compact case, we define the {\em indicial radius} of $P$ as the smallest
non-negative number $R$ such that $P$ has a characteristic exponent
whose real part is $\tfrac{n-1}{2} -k + R$. 
 
\begin{theorem} \label{thmFredholm}
Let $(M,g)$ be a connected asymptotically hyperbolic $n$-manifold of
class $C^{l,\beta}_\tau$, with $n\geq 2$, $l\geq 2$, $0\leq \beta< 1$, and
$\tau>0$ and let $E\rightarrow M$ be a geometric tensor bundle over $M$.
Suppose that $P: C^\infty(M;E) \rightarrow C^\infty(M;E)$ is an elliptic,
formally self-adjoint, geometric partial differential operator of order
$m$, $0 < m \leq l$, and assume that there exists a compact set
$K\subset M$ and a positive constant $C$ such that 
\begin{equation} \label{EstimateInfinity}
\|u\|_{L^2} \leq C \|Pu\|_{L^2}
\end{equation}
for all $u\in C^\infty_c(M\setminus K;E)$. Let $R$ be the indicial
radius of $P$. 
\begin{itemize}
\item
If $1<p<\infty$ and $m\leq k\leq l$ then the natural extension 
\[
P:W^{k,p}_{\delta}(M;E)\rightarrow W^{k-m,p}_\delta(M;E)
\]
is Fredholm for
$|\delta + \tfrac{n-1}{p} - \tfrac{n-1}{2}|<R$.
In that case, its index is zero, and its kernel is equal to the $L^2$ kernel
of $P$.
\item
If $0<\alpha<1$ and $m< k+\alpha \leq l+\beta$ then the natural extension 
\[
P:C^{k,\alpha}_{\delta}(M;E)\rightarrow C^{k-m,\alpha}_\delta(M;E)
\]
is Fredholm for $|\delta-\tfrac{n-1}{2}|<R$. In that case, its index
is zero, and its kernel is equal to the $L^2$ kernel of $P$.
\end{itemize}
\end{theorem}

\begin{proof}
The proof goes as in \cite[Chapter~6]{LeeFredholm}, except for the steps
which explicitly use coordinates at infinity. We verify that these steps
can be carried out in our setting, that is for asymptotically hyperbolic
manifolds as in Definition~\ref{defAHmanifold}. Specifically, we need to
adapt the construction of a parametrix given in Proposition~6.2 and
Corollary 6.3 of \cite{LeeFredholm}. In fact, the adaptation turns out to
be rather straightforward since we have a single chart at infinity naturally
replacing (finitely many) boundary  M\"obius charts of \cite[Chapter~6]{LeeFredholm}.

Let $\Psi: M\setminus K_0 \to \bH^n \setminus B_{R_0}$ be a chart at
infinity as in Definition~\ref{defAHmanifold}. As in
\cite[Appendix A]{NonCMC} we use this chart to construct a bundle
$\Ebrev \to \bH^n$ which is defined using the same $O(n)$-representation
as the one which defines $E$, and an isomorphism
$\Upsilon: \Ebrev|_{\bH^n \setminus B_{R_0}} \to E|_{M\setminus K_0}$. The
isomorphism $\Upsilon$, its inverse and their first $l$ derivatives all
have uniformly bounded norms on $\bH^n \setminus B_{R_0}$, respectively
$M\setminus K_0$. Let
$\Pbrev: C^\infty (\bH^n;\Ebrev) \to C^\infty (\bH^n;\Ebrev) $ be the
operator on hyperbolic space with the same local coordinate expression
as $P$. We define $P': C^\infty (\bH^n \setminus B_{R_0};\Ebrev) \to
C^\infty (\bH^n \setminus B_{R_0};\Ebrev) $ by
\[
P'u = \Upsilon^{-1} P \Upsilon u. 
\]
Let $R_1 \geq R_0$. Since $P$ is a geometric operator and $g$ is
$C^{l,\beta}_\tau$-asymptotically hyperbolic, we conclude that for each
$\delta\in \bR$, $0 < \alpha < 1$, $1 < p < \infty$, and $k$ such that
$m \leq k \leq l$ and $m < k + \alpha \leq l+\beta$ there exists a
positive constant $C$ independent of $R_1$ such that
\begin{equation} \label{eqSmallDiff1}
\|P'u - \Pbrev u\|_{C^{k-m,\alpha}_\delta(\bH^n \setminus B_{R_1};\Ebrev)}
< C e^{-\tau R_1}\|u\|_{C^{k,\alpha}_\delta(\bH^n \setminus B_{R_1};\Ebrev)}
\end{equation}
holds for all $u \in C^{k,\alpha}_\delta (\bH^n \setminus B_{R_0};\Ebrev)$, and
\begin{equation} \label{eqSmallDiff2}
\|P'u - \Pbrev u\|_{W^{k-m,p}_\delta(\bH^n \setminus B_{R_1};\Ebrev)}
< C e^{-\tau R_1} \|u\|_{W^{k,p}_\delta(\bH^n \setminus B_{R_1};\Ebrev)}
\end{equation}
holds for all $u \in W^{k,p}_\delta (\bH^n \setminus B_{R_0};\Ebrev)$.

Now suppose that $P$ satisfies \eqref{EstimateInfinity}. Then, by the
properties of $\Upsilon$, the operator $P'$ also satisfies 
\eqref{EstimateInfinity} (perhaps with a larger constant). Consequently, if $R_1$ is sufficiently large, it 
follows by \eqref{eqSmallDiff2} and standard elliptic regularity that 
$\Pbrev$ satisfies $\|u\|_{L^2}\leq C \|\Pbrev u\|_{L^2}$ for all
$u\in C^\infty_c(\bH^n \setminus B_{R_1};\Ebrev)$, possibly with a larger
value of $C$ than in \eqref{EstimateInfinity}. By
\cite[Theorems 5.7 and 5.9]{LeeFredholm} we conclude that $\Pbrev$ is
invertible as an operator
$W^{k,p}_\delta (\bH^n;\Ebrev) \to W^{k-m}_\delta(\bH^n;\Ebrev)$ for
$|\delta + \tfrac{n-1}{p} - \tfrac{n-1}{2}|<R$ and as an operator
$C^{k,\alpha}_{\delta}(\bH^n;\Ebrev)\rightarrow C^{k-m,\alpha}_\delta(\bH^n;\Ebrev)$
for $|\delta - \tfrac{n-1}{2}|<R$.

Now assume that $R_1\geq R_0$ is sufficiently large and let $K_N$ be such
that $M\setminus K_N= \Psi^{-1} (\bH^n \setminus B_{N R_1})$ for $N=1,2,\ldots$. 
We define two smooth bump functions: $\psi_0$ equal to 1 on $ K_2$  and supported on $K_4$ and 
$\psi_1$ equal to 1 on $M\setminus K_2$ and supported on $M\setminus K_1$. Let
\begin{equation*}
\phi = \frac{\psi_1}{\sqrt{\psi_0^2 + \psi_1^2}},
\end{equation*}
so that $\{1- \phi^2,\phi^2\}$ is a partition of unity subordinate to the cover $\{K_4,M\setminus{K_1}\}$. 
Clearly, $\phi \in C^{l,\beta}(M)$.

We proceed by
defining the operators
$Q, S: C^\infty_c(M, E)\to C^\infty_c(M, E)$ by 
\begin{align*}
Q u &=\phi  \Upsilon \Pbrev^{-1} \Upsilon^{-1} (\phi u), \\
S u &= \phi \Upsilon \Pbrev^{-1}(P'-\Pbrev) \Upsilon^{-1} (\phi u),\\
T u &= \phi \Upsilon \Pbrev^{-1} \Upsilon^{-1} ([\phi, P]u).
\end{align*}
A straightforward computation as in \cite[Proof of Proposition 6.2]{LeeFredholm} shows that
\[
QPu = u + S u + T u.
\]
Furthermore, it follows from the above discussion that $Q$, $S$, and $T$ extend to bounded maps
\begin{align*}
Q &: W^{0,p}_\delta (M\setminus K_1;E) \to W^{m,p}_\delta (M\setminus K_1;E),\\
S &: W^{m,p}_\delta (M\setminus K_1;E) \to W^{m,p}_\delta (M\setminus K_1;E),\\
T &:  W^{m-1,p}_\delta (M\setminus K_1;E) \to W^{m,p}_\delta (M\setminus K_1;E)
\end{align*}
for $|\delta + \frac{n-1}{p} - \tfrac{n-1}{2}|<R$, and to bounded maps
\begin{align*}
Q &: C^{0,\alpha}_\delta (M\setminus K_1;E) \to C^{m,\alpha}_\delta (M\setminus K_1;E),\\
S &: C^{m,\alpha}_\delta (M\setminus K_1;E) \to C^{m,\alpha}_\delta (M\setminus K_1;E),\\
T &:  C^{m-1,\alpha}_\delta (M\setminus K_1;E) \to C^{m,\alpha}_\delta (M\setminus K_1;E)
\end{align*}
for $|\delta - \tfrac{n-1}{2}|<R$.
In particular, as a consequence of \eqref{eqSmallDiff1} and \eqref{eqSmallDiff2},
we see that if $u$ is supported in $M\setminus K_1$, then 
\[
\|S u\|_{W^{m,p}_\delta} \leq C e^{-\tau R_1} \|u\|_{W^{m,p}_\delta},
\quad
\|S u\|_{C^{m,\alpha}_\delta} \leq C e^{-\tau R_1} \|u\|_{C^{m,\alpha}_\delta}
\]
holds for some constant $C$ independent of $R_1$ and $u$. Without loss of
generality we may assume that $Ce^{-\tau R_1} < \tfrac{1}{2}$, and it
follows that the operators 
\begin{align*}
\operatorname{Id} + S
&:
W^{m,p}_\delta (M\setminus K_1;E) \to W^{m,p}_\delta (M\setminus K_1;E), \\
\operatorname{Id} + S
&:
C^{m,\alpha}_\delta (M\setminus K_1;E) \to C^{m,\alpha}_\delta (M\setminus K_1;E)
\end{align*}
have bounded inverses. This implies that, whenever $u$ has support in
$M\setminus K_1$, we have
\[
\widetilde{Q} P u = u + \widetilde{T} u,
\]
where $\widetilde{Q}=(\operatorname{Id} + S)^{-1} \circ Q$ is bounded as
an operator
\begin{align*}
\widetilde{Q} &: W^{0,p}_\delta (M\setminus K_1;E) \to W^{m,p}_\delta (M\setminus K_1;E),\\
\widetilde{Q} &: C^{0,\alpha}_\delta (M\setminus K_1;E) \to C^{m,\alpha}_\delta (M\setminus K_1;E)
\end{align*}
and $\widetilde{T}=(\operatorname{Id} + S)^{-1} \circ T$ is bounded as
an operator
\begin{align*}
\widetilde{T} &: W^{m-1,p}_\delta (M\setminus K_1;E) \to W^{m,p}_\delta (M\setminus K_1;E), \\
\widetilde{T} &: C^{m-1,\alpha}_\delta (M\setminus K_1;E) \to C^{m,\alpha}_\delta (M\setminus K_1;E)
\end{align*}
where $\delta$ is in the same range as above. As a consequence of this
parametrix construction, improved elliptic regularity results
\cite[Proposition~6.5]{LeeFredholm} hold for asymptotically hyperbolic
manifolds as in Definition~\ref{defAHmanifold}. 

The rest of the proof does not use coordinates at infinity, and the
reader is referred to \cite{LeeFredholm} for details.
\end{proof}

\begin{proposition} \label{propModelOperators}
The operator $\Delta - n$ and the vector Laplacian $\Delta_L$ satisfy
the conditions of Theorem~\ref{thmFredholm} with $R=\tfrac{n+1}{2}$.
\end{proposition}

\begin{proof}
It is not complicated to check that the $L^2$-estimate at infinity \eqref{EstimateInfinity} holds for $\Delta - n$. For $\Delta_L$ the $L^2$-estimate at infinity can be proven by standard methods, see for example \cite[Section 7]{LeeFredholm} or Appendix~B in \cite{NonCMC}. The critical exponents of both operators can be computed using the explicit expressions for their components, see the proof of Proposition~\ref{propCritical} below.
\end{proof}

\section{Solutions of critical order}
\label{secAbove} 

Suppose that $P:C^\infty (M;E)\to C^\infty(M;E)$ is a formally self-adjoint 
geometric elliptic operator of order $m$ satisfying the conditions of 
Theorem~\ref{thmFredholm} and let $\delta_-<\delta_+$ be its critical
exponents. Roughly speaking, if 
$u=O(e^{-\delta r})$ for some $\delta \in (\delta_- , \delta_+)$
then $Pu = O(e^{-\kappa r})$ for some $\kappa \in (\delta_-,\delta_+)$
implies that $u=O(e^{-\kappa r})$, 
see \cite[Proposition~6.5]{LeeFredholm}. At the same time,
$Pu = O(e^{-\delta_+ r})$ does not necessarily imply $u=O(e^{-\delta_+ r})$.
An extensive study of the asymptotic behaviour of solutions outside of the
Fredholm interval can be found in \cite{ACF} and \cite[Chapter~4]{AnderssonChrusciel} 
in the case of conformally compact metrics. Analogous results can be proven
for asymptotically hyperbolic manifolds as in
Definition~\ref{defAHmanifold} using the following simple lemma. 
\begin{lemma} \label{lemODE}
Consider the ordinary differential equation 
\begin{equation} \label{eqODE}
u'' + A u' + B u = f.
\end{equation}
Assume that $A^2 - 4 B > 0$ so that the characteristic equation 
$\lambda^2 - A \lambda + B = 0$
has two distinct real roots 
$\delta_- < \delta_+$. Suppose that \eqref{eqODE} holds for 
$u = u(r) = O(e^{-\delta r})$ and $f =f(r) = O(e^{-\kappa r})$ 
for some $\kappa > \delta_+$. Then 
\begin{itemize}
\item $\delta > \delta_-$ implies that $u = O(e^{-\delta_+ r})$, and 
\item $\delta > \delta_+$ implies that $u = O(e^{-\kappa r})$. 
\end{itemize}
\end{lemma}
Note that we use a possibly non-standard characteristic equation
which results from substituting $u=e^{-\lambda r}$ rather than
$u=e^{\lambda r}$ into \eqref{eqODE}.
\begin{proof}
This is a consequence of the explicit formula
\begin{equation} \label{eqVarPar}
\begin{split}
u
&=
\Lambda_- e^{-\delta_- r} + \Lambda_+ e^{-\delta_+ r} \\
&\qquad
- \frac{1}{\delta_+ -\delta_-}
\left(e^{-\delta_-r} \int_r^\infty e^{\delta_- s} f(s) \, ds
- e^{-\delta_+r} \int_r ^\infty e^{\delta_+s} f(s)\, ds \right)
\end{split}
\end{equation}
for the solutions of \eqref{eqODE}. Note that $\Lambda_-$ and $\Lambda_+$ 
do not depend on $r$.
\end{proof}

In this paper we use the following result.

\begin{proposition} \label{propCritical}
Let $(M,g)$ be a connected asymptotically hyperbolic $n$-manifold of
class $C^{l,\beta}_\tau$, with $n\geq 2$, $l\geq 2$, $0 \leq \beta< 1$, and 
$\tau>0$.
\begin{itemize}
\item 
Assume that $v \in C^{0,0}_\delta$ is such that 
$\Delta v - n v \in C^{k-2,\alpha}_{n + \varepsilon}$ for $\varepsilon > 0$,
$0 < \alpha < 1$, and $k+\alpha \leq l + \beta$.
If $\delta > -1$ then $v \in C^{k,\alpha}_n$.
If $\delta > n$, then $v \in C^{k,\alpha}_{n+\varepsilon}$. 
\item 
Assume that $Z \in C^{0,0}_\delta$ is such that 
$\Delta_L Z \in C^{k-2,\alpha}_{n + \varepsilon}$ for $\varepsilon > 0$,
$0 < \alpha < 1$, and $k+\alpha \leq l + \beta$.
If $\delta > -1$ then $Z \in C^{k,\alpha}_n$. 
If $\delta > n$, then $Z \in C^{k,\alpha}_{n+\varepsilon}$. 
\end{itemize}
\end{proposition}

\begin{proof}
We first prove the second claim. A straightforward computation shows
that if $Z\in C^{2,\alpha}_{\delta'}$ then
\begin{align*}
(\Delta_L Z)_r 
&= \tfrac{2(n-1)}{n}(\partial^2_{rr} Z_r + (n-1) \partial_r Z_r - n Z_r) 
+ O(e^{-(\delta' + \gamma)r}), \\
(\Delta_L Z)_\psi 
&= \partial^2_{rr} Z_\psi + (n-3) \partial_r Z_\psi - 2(n-1) Z_\psi 
+ O(e^{-(\delta' + \gamma -1)r}) 
\end{align*}
for $\gamma = \min\{1,\tau\}$. Note that in our case
$Z \in C^{k,\alpha}_{\delta'}$ for any $\delta' \in (-1,n)$ as a consequence
of improved elliptic regularity \cite[Proposition~6.5]{LeeFredholm} so
we may assume that $\delta'+\gamma > n$. Hence the components of $Z$
satisfy 
\begin{align*}
\partial^2_{rr} Z_r + (n-1) \partial_r Z_r - n Z_r 
&= O(e^{-\kappa r}), \\
\partial^2_{rr} Z_\psi + (n-3) \partial_r Z_\psi - 2(n-1) Z_\psi 
&= O(e^{-(\kappa-1) r})
\end{align*}
for $\kappa = \min\{\delta' + \gamma, n + \varepsilon\}$. From 
Lemma~\ref{lemODE} it follows that $Z_r = O(e^{-n r})$, and
$Z_\psi = O(e^{-(n-1)r})$, hence $Z \in C^{k,\alpha}_{n}$ by standard
elliptic regularity  \cite[Lemma~4.8]{LeeFredholm}. Similarly, if $\delta>n$ it follows that
$Z \in C^{k,\alpha}_{n+\varepsilon}$, possibly after repeating this 
argument finitely many times in order to ensure that
$\kappa = n + \varepsilon$.

The first claim is proven similarly using Lemma~\ref{lemODE} and the 
fact that
\[
\Delta v -n v = 
\partial^2_{rr} v + (n-1) \partial_r v - n v + O(e^{-(\delta' + \gamma)r})
\]
for $\gamma = \min\{1,\tau\}$ when $v \in C^{2,\alpha}_{\delta'}$.
\end{proof}

\section{On the unique continuation property}
\label{secUC}

The following result is a straightforward consequence of the unique
continuation results by Mazzeo \cite[Theorem~7]{Mazzeo} and Kazdan
\cite[Theorem~1.8]{Kazdan}.

\begin{proposition}
Let $(M,g)$ be a $C^{2,\beta}_\tau$-asymptotically hyperbolic manifold for 
$\tau>0$ and $0\leq \beta <1$, and let $E$ be a geometric tensor bundle
over $M$. Suppose that $u\in C^2(M;E)$ satisfies the differential
inequality
\begin{equation} \label{eqIneqUC1}
|\Delta u| \leq C (|u| + |\nabla u|),
\end{equation}
where $\Delta = -\nabla^*\nabla$ is the rough Laplacian. If $u$ vanishes
to infinite order at infinity, that is $|u|=O(e^{-Nr})$ for any $N>0$,
then $u=0$ on $M$.
\end{proposition}
\begin{proof}
The hyperbolic metric $b=dr^2 + \sinh^2 r \, \sigma$ clearly satisfies the
conditions (4)--(6) in \cite{Mazzeo}. We may therefore combine Theorem~7
in this reference with the fact that $g-b\in C^{2,\beta}_\tau$ to conclude
that for any $z\in C^2(M;E)$ vanishing on $\{r\leq r_0\}$ and to infinite
order at infinity we have
\begin{equation} \label{eqIneqUC3}
t^3 \int_M e^{2tr} |z|^2 \, d\mu^g + t \int_M e^{2tr} |\nabla z|^2 \, d\mu^g
\leq
C_0 \int_M e^{2tr} |\Delta z|^2 \, d\mu^g.
\end{equation}
Here it is assumed that $t$ and $r_0$ are sufficiently large, and that
$C_0$ does not depend on $t$. We now argue as in
\cite[Corollary 11]{Mazzeo} and set $z=\phi u$ where $\phi$ vanishes on
$\{r\leq r_0\}$, and is equal to $1$ on $\{r\geq r_0 + 1\}$. As a
consequence of \eqref{eqIneqUC3} combined with \eqref{eqIneqUC1} we obtain
\[
(t^3-2 C_0 C^2) \int_{r_0+1} ^\infty e^{2tr} |u|^2 \, d\mu^g
\leq
C_0 \int_{r_0}^{r_0+1} e^{2tr} |\Delta z|^2 \, d\mu^g.
\]
When $t\to\infty$ the left hand side is at least of order
$O(t^3 e^{2(r_0+1) t})$, whereas the right hand side has order
$O(e^{2(r_0+1) t})$. Hence $u=0$ on $\{r\geq r_0 + 1\}$. To conclude the
proof, it suffices to note that $u$ satisfies the conditions of the
strong unique continuation theorem \cite[Theorem~1.8]{Kazdan}, thus
$u=0$ on $M$.
\end{proof}

\bibliographystyle{amsplain}
\bibliography{biblio}

\end{document}